\documentclass[11pt, reqno]{amsart}
\usepackage{amsmath,amssymb}
\usepackage[margin=1.25in]{geometry}
\usepackage{graphicx}
\def\COMMENT#1{}
\def\TASK#1{}
\def\noproof{{\unskip\nobreak\hfill\penalty50\hskip2em\hbox{}\nobreak\hfill%
        $\square$\parfillskip=0pt\finalhyphendemerits=0\par}\goodbreak}
\def\endproof{\noproof\bigskip}
\newdimen\margin   
\def\textno#1&#2\par{%
    \margin=\hsize
    \advance\margin by -4\parindent
           \setbox1=\hbox{\sl#1}%
    \ifdim\wd1 < \margin
       $$\box1\eqno#2$$%
    \else
       \bigbreak
       \hbox to \hsize{\indent$\vcenter{\advance\hsize by -3\parindent
       \sl\noindent#1}\hfil#2$}%
       \bigbreak
    \fi}
\def\proof{\removelastskip\penalty55\medskip\noindent{\bf Proof. }}
\def\C{\mathcal{C}}
\def\eps{\varepsilon}

\def\T{\mathcal{T}}
\def\K{\mathcal{K}}

\newcommand{\ep}{\varepsilon} 
\newcommand{\U}{\mathcal{U}}
\newcommand{\cD}{\mathcal{D}}
\newcommand{\floor}[1]{\left\lfloor#1\right\rfloor}
\newcommand{\ceiling}[1]{\left\lceil#1\right\rceil}

\newcommand{\Kb}{\bar{\mathcal K}}
\newcommand{\N}{\mathbb{N}}

\def\T{\mathcal{T}}
\def\K{\mathcal{K}}
\def\B{\mathcal{B}}
\def\C{\mathcal{C}}
\def\D{\mathcal{D}}
\def\M{\mathcal{M}}

\newtheorem{theorem}{Theorem}[section]
\newtheorem{corollary}[theorem]{Corollary}
\newtheorem{proposition}[theorem]{Proposition}
\newtheorem{conjecture}[theorem]{Conjecture}
\newtheorem{observation}[theorem]{Observation}
\newtheorem{definition}[theorem]{Definition}
\newtheorem{question}[theorem]{Question}
\newtheorem{claim}[theorem]{Claim}
\newtheorem{lemma}[theorem]{Lemma}

\newtheorem{fact}[theorem]{Fact}

\begin{document}
\title{Tiling directed graphs with tournaments}
\author{Andrzej Czygrinow, Louis DeBiasio, Theodore Molla and Andrew Treglown}
\thanks{The second author is supported by Simons Foundation Grant \# 283194. The third author is supported by NSF Grant DMS-1500121.  The fourth author is supported by EPSRC grant EP/M016641/1.}
\date{\today}

\label{firstpage}
\begin{abstract} 
  The Hajnal--Szemer\'edi theorem states that 
  for any integer $r \ge 1$ and any multiple $n$ of $r$,
  if $G$ is a graph on $n$ vertices and $\delta(G) \ge (1 - 1/r)n$,
  then $G$ can be partitioned into $n/r$ vertex-disjoint copies of the
  complete graph on $r$ vertices.
  We prove a very general analogue of this result for directed graphs:
  for any integer $r \ge 4$ and any sufficiently large multiple $n$ of $r$,
  if $G$ is a directed graph on $n$ vertices and every vertex is incident to at
  least $2(1 - 1/r)n - 1$ directed edges, then $G$ can be partitioned into
  $n/r$ vertex-disjoint subgraphs of size $r$ each of which contain
  every tournament on $r$ vertices. A related Tur\'an-type result is also proven.
\end{abstract}
\maketitle

\section{Introduction}
\subsection{Tilings in graphs}
Given two (di)graphs $H$  and $G$, an \emph{$H$-tiling} in $G$ 
is a collection of vertex-disjoint copies of $H$ in $G$. An
$H$-tiling is called \emph{perfect} if it covers all the vertices of $G$.
Perfect $H$-tilings are also referred to as \emph{$H$-factors} or \emph{perfect $H$-packings}. 
If $H$ is connected and of order at least three,
the problem of deciding
whether a graph $G$ contains a perfect $H$-tiling is NP-complete~\cite{hell}.
In light of this, it is natural  to ask for simple sufficient conditions which force a graph
to contain a perfect $H$-tiling. 

A cornerstone result in extremal graph theory is the following  theorem of Hajnal and Szemer\'edi~\cite{hs}.
\begin{theorem}[Hajnal and Szemer\'edi~\cite{hs}]\label{hs}
Every graph $G$ whose order $n$
is divisible by $r$ and whose minimum degree satisfies $\delta (G) \geq (1-1/r)n$ contains a perfect $K_r$-tiling. 
\end{theorem}
Notice that the minimum degree condition in Theorem~\ref{hs} is tight. Earlier, Corr\'adi and Hajnal~\cite{corradi} proved Theorem~\ref{hs} in the
case when $r=3$. More recently, Kierstead and Kostochka~\cite{short} gave a short proof of the Hajnal--Szemer\'edi theorem.

Over the last three decades there has been much work on generalising the Hajnal--Szemer\'edi theorem. One highlight in this direction is a result of 
 K\"uhn and Osthus~\cite{kuhn, kuhn2} that characterises, up to an additive constant, the minimum degree which ensures that a graph $G$ 
contains a perfect $H$-tiling for an \emph{arbitrary} graph $H$. Other notable results include an  \emph{Ore-type} analogue of
the Hajnal--Szemer\'edi theorem of Kierstead and Kostochka~\cite{kier} and an \emph{$r$-partite} version of the Hajnal--Szemer\'edi theorem of Keevash and Mycroft~\cite{my}.
See~\cite{survey} for a survey including many of the results on \emph{graph} tiling. 

There has also been  interest  in tiling problems for directed graphs and hypergraphs. A recent survey of Zhao~\cite{zsurvey} gives an extensive overview of the latter problem.
In this paper we prove a directed analogue of the Hajnal--Szemer\'edi theorem.

\subsection{Tilings  in directed graphs}\label{sec1}
Throughout this paper, the digraphs we consider do not have loops and we allow for at most one edge in each direction between any pair of vertices. An oriented graph is a digraph without 
$2$-cycles. 

For digraphs there is more than one natural notion of degree: The \emph{minimum semidegree} $\delta ^0 (G)$ of a digraph $G$ is the minimum of its minimum outdegree $\delta ^+ (G)$ and its
minimum indegree $\delta ^- (G)$.  The \emph{minimum degree} $\delta (G)$ of $G$ 
is  the minimum number of edges incident to a vertex in $G$.

For oriented graphs, there has been some progress on obtaining degree conditions that force tilings. Denote by $T_r$ the transitive tournament of $r$ vertices and by $C_3$ the cyclic triangle.
Yuster~\cite{yuster} observed that an oriented graph $G$ on $n\in 3\mathbb N$ vertices and with $\delta (G) \geq 5n/6$ contains a perfect $T_3$-tiling (and also gave a minimum degree condition which forces a perfect $T_r$-tiling for $r >3$).
More recently, Balogh, Lo and Molla~\cite{blm} determined the minimum semidegree threshold which ensures a perfect $T_3$-tiling in an oriented graph, thereby resolving a conjecture from~\cite{problem}.
Keevash and Sudakov~\cite{keevs} showed that every oriented graph $G$ on $n$ vertices  with $\delta ^0 (G) \geq (1/2-o(1))n$ contains a $C_3$-tiling covering all but at most $3$ vertices. (There are examples that show even 
$\delta ^0 (G) \geq  (n-3)/2$ does not guarantee a perfect $C_3$-tiling.)
 
Denote by $\mathcal T_r$ the set of all tournaments on $r$ vertices. Let $T \in \mathcal T_r$.
For digraphs, the minimum semidegree threshold that forces a perfect $T$-tiling was characterised in~\cite{treg}.
\begin{theorem}\label{tregthm}\cite{treg} Given an integer $r \geq 3$, there exists an $n_0 \in \mathbb N$ such that the following holds.
Suppose $T \in \mathcal T_r$ and $G$ is a digraph on $n \geq n_0$ vertices where $r$ divides $n$.
If $$\delta ^0 (G)\geq (1-1/r)n$$ then $G$ contains a perfect $T$-tiling.
\end{theorem}
Notice that the minimum semidegree condition in Theorem~\ref{tregthm} is tight.
Note also that Theorem~\ref{tregthm} implies the Hajnal--Szemer\'edi theorem for large graphs. An earlier result of Czygrinow, Kierstead and Molla~\cite{ckm} gives an asymptotic version of Theorem~\ref{tregthm}
for perfect $C_3$-tilings.

It is natural to ask whether Theorem~\ref{tregthm} can be strengthened by replacing the \emph{minimum semidegree} condition with a \emph{minimum degree} condition. In particular, can one replace the minimum semidegree condition in Theorem~\ref{tregthm} with
$\delta (G) \geq 2(1-1/r)n-1$? However, when $T=C_3$ the answer is no. Indeed, an example of Wang~\cite{wang} shows that $\delta (G) \geq ({3n-5})/{2}$ does not ensure a perfect $C_3$-tiling. On the other hand, he  showed that minimum degree  
$\delta (G) \geq {(3n-3)}/{2}$ does force a perfect $C_3$-tiling in a digraph $G$. This led to the following question being raised in~\cite{treg}.
\begin{question}\label{ques1} 
  Let $n , r \in \mathbb N$ such that $r$ divides $n$. Let $T \in \mathcal T_r \setminus \{C_3\}$.
  Does every digraph $G$ on $n$ vertices with $\delta (G) \geq 2(1-1/r)n-1$ contain a perfect $T$-tiling?
\end{question}
 Czygrinow, DeBiasio,  Kierstead and  Molla~\cite{cdkm} answered Question~\ref{ques1} in the affirmative for perfect $T_r$-tilings and also in the case when

The main result of this paper gives an exact solution to a stronger version of
Question~\ref{ques1} for \emph{all} $r \geq 4$.
\begin{theorem}\label{mainthm}
Given an integer $r \geq 4$, there exists an $n_0 \in \mathbb N$ such that the following holds.
Suppose $G$ is a digraph on $n \geq n_0$ vertices where $r$ divides $n$.
If $$\delta(G)\geq 2(1-1/r)n-1,$$ then $G$ contains $n/r$ vertex-disjoint subdigraphs each of which contains every tournament on $r$ vertices. 
\end{theorem}

The following theorem from \cite{ckm}, in some sense, answers the analogous question when $r=3$.
\begin{theorem}\label{mainthmr3}\cite{ckm}
  Suppose $G$ is a digraph on $n$ vertices where $3$ divides $n$.
  If $\delta(G) \ge 4n/3 -1$, then there exist $n/3$ vertex-disjoint subdigraphs each of size $3$
  such that each subdigraph contains a $T_3$ and all but at most one contains $C_3$ as well.
\end{theorem}
This is best possible in the following two senses: (i) there exist digraphs $G$ for which
$\delta(G) = 4n/3 - 2$ and which do not contain a triangle factor of \emph{any} kind
and; (ii) by Wang's example in~\cite{wang}, 
there exist digraphs such that $\delta(G) \ge 4n/3 - 1$ that do not have a perfect $C_3$-tiling.
However, there is perhaps more to say about the case when $r=3$,
as the following conjecture, which originally appeared in \cite{molla13phd}, suggests.
\begin{conjecture}\label{conj:2strong}\cite{molla13phd}
  Suppose $G$ is a digraph on $n$ vertices where $3$ divides $n$.
  If $\delta(G) \ge 4n/3 -1$ and $G$ is strongly $2$-connected, 
  then there exist $n/3$ vertex-disjoint subdigraphs 
  such that each of these subdigraphs contain both $T_3$ and $C_3$. 
\end{conjecture}
It should be noted that in~\cite{ckm} Conjecture~\ref{conj:2strong}
was proven when $\delta(G) \ge (3n - 3)/2$.  Note that
when $\delta(G) \ge (3n - 3)/2$, $G$ is strongly $2$-connected.

\subsection{Tilings in multigraphs}
Instead of proving Theorem~\ref{mainthm} directly, we will prove a more general result concerning tilings in multigraphs. 
A similar approach was taken in \cite{ckm} and \cite{cdkm}.

Suppose that $M$ is a multigraph. The \emph{minimum degree} $\delta (M)$ of $M$ 
is  the minimum number of edges incident to a vertex in $M$. For $x,y \in V(M)$ we write $\mu(xy)$ to denote the number of edges between  $x$ and $y$ in $M$. We say that a loopless multigraph $M$ is \emph{standard}
if $\mu (xy) \leq 2$ for all $x,y \in V(M)$. Given vertices $x,y$ in a standard multigraph $M$ we say that $xy$ is a \emph{light edge} if $\mu(xy)=1$ and a \emph{heavy edge} if $\mu (xy)=2$.

Given a digraph $G$, the \emph{underlying multigraph} $M$ of $G$ is the standard multigraph obtained from $G$ by ignoring the orientations of edges. Given a standard multigraph $M$, an orientation of the edges is \emph{legal}
if the resulting graph $G$ is a digraph (i.e. there is at most one edge in each direction between any pair of vertices in $G$). A standard multigraph $M$ on $r$ vertices is \emph{universal} if, given any legal orientation $G$ of $M$, we have that $T \subseteq G$
for \emph{every} $T \in \mathcal T_r$. 
For example, let $M$ be a standard multigraph on $r$ vertices where $\mu(xy)=1,2$ for all distinct $x, y \in V(M)$ and the collection of light edges in $M$ forms a matching. Then $M$ is universal.
On the other hand if $M$ is a standard multigraph on $n$ vertices that contains a cycle\COMMENT{LD: Changed ``triangle'' to ``cycle''} on light edges then $M$ is not universal. (There is a legal orientation of $M$ without a copy of $T_r$.)
Write $\mathcal U_r$ for the set of all universal standard multigraphs on $r$ vertices. 

Given a collection of (multi)graphs $\mathcal X$, an \emph{$\mathcal X$-tiling} in a (multi)graph $M$ 
is a collection of vertex-disjoint copies of elements of $\mathcal X$ in $M$.
 An
$\mathcal X$-tiling is called \emph{perfect} if it covers all the vertices of $M$.
We refer to the elements of an $\mathcal X$-tiling as \emph{tiles}.
The next result (originally conjectured in \cite{cdkm}) ensures a standard multigraph of high minimum degree contains a perfect $\mathcal U_r$-tiling.
\begin{theorem}\label{mainthm2}
Given an integer $r \geq 4$, there exists an $n_0 \in \mathbb N$ such that the following holds.
Suppose  $M$ is a standard multigraph on $n \geq n_0$ vertices where $r$ divides $n$.
If $$\delta  (M)\geq 2(1-1/r)n-1$$ then $M$ contains a perfect $\mathcal U_r$-tiling.
\end{theorem}
Notice that Theorem~\ref{mainthm2}  implies Theorem~\ref{mainthm}. 
\COMMENT{AT: should we mention that Theorem~\ref{mainthm2} was conjectured in~\cite{cdkm}? LD: I made note of this above}

\begin{figure}[ht]
\includegraphics[scale=.55]{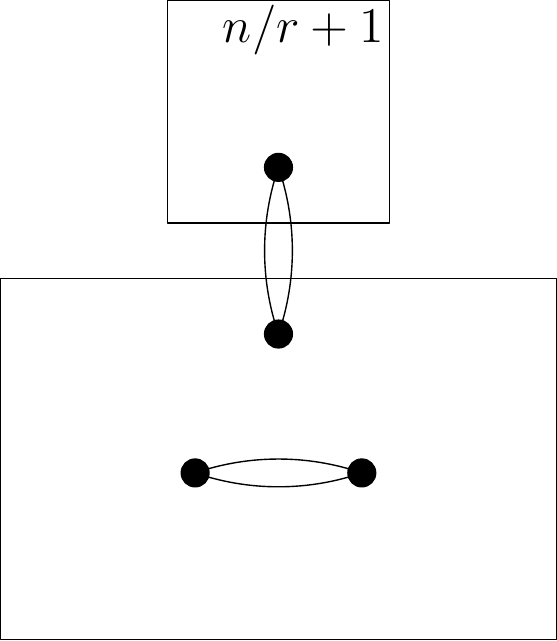}\hspace{.5cm}
\includegraphics[scale=.55]{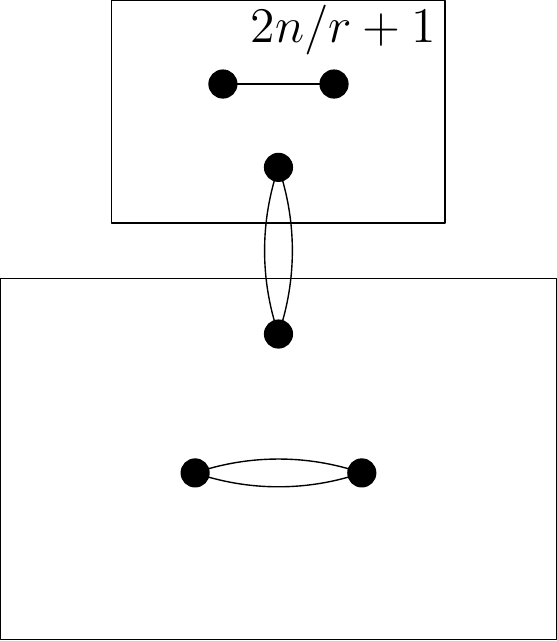}\hspace{.5cm}
\includegraphics[scale=.55]{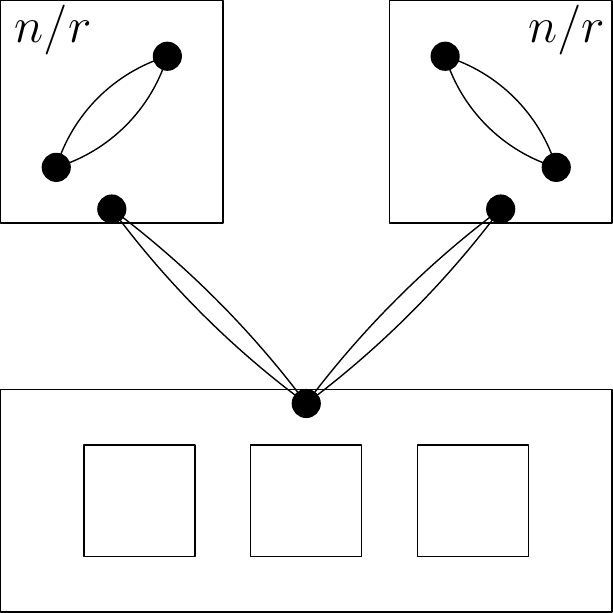}\hspace{.5cm}
\includegraphics[scale=.55]{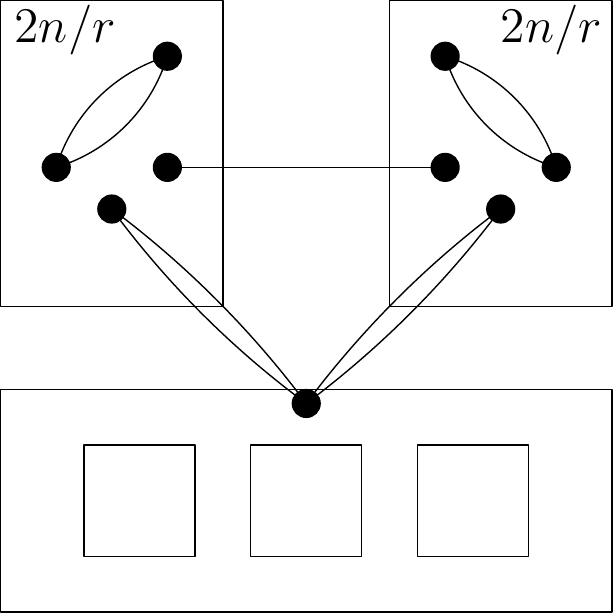}
\caption{From left to right, the tightness examples $M_1, M_2, M_3, M_4$}\label{fig:extremalcases}
\end{figure}

There are four different examples which show that the minimum degree condition in Theorem~\ref{mainthm2} is tight.
Let $M_1$ and $M_2$ be standard multigraphs on $n$ vertices which contain all possible edges except that $M_1$ contains an independent set $U$ of size $n/r + 1$ 
and  $M_2$ contains a set $U$ of size $2n/r + 1$ such that between any two distinct vertices $u,v \in U$
$\mu(uv) = 1$.
For both $i=1,2$, $\delta(M_i) = 2(1 - 1/r)n - 2$, but $M_i$ does not contain a perfect $\mathcal U_r$-tiling.
In the case of $M_1$, this is because every element in $\mathcal U_r$ has at most $1$ vertex in $U$
and in the case of $M_2$, it is because every element in $\mathcal U_r$ has at most $2$ vertices in $U$.

Suppose $n/r$ is odd.  We define the standard multigraph $M_3$ on $n$ vertices\COMMENT{AT NEW:reworded} as follows: Take two disjoint sets $X, Y$ of size $n/r$.  Inside the sets $X, Y$ place all heavy edges, and between $X$ and $Y$ place no edges.  From $X\cup Y$ to the remaining vertices, place all heavy edges.  Now partition the remaining $(1-2/r)n$ vertices into sets of size $n/r$ or $2n/r$.  Between all such sets, place all heavy edges.  Inside the sets of size $2n/r$ place all light edges and inside the sets of size $n/r$ place no edges.  We have $\delta(M_3)=2(1-1/r)n-2$.  If $M_3$ contained a perfect $\U_r$-tiling, each copy of $\U_r$ would intersect the sets of size $n/r$ from $V(M_3)\setminus (X\cup Y)$ in exactly one vertex and the sets of size $2n/r$ in exactly two vertices, and furthermore every copy of $\U_r$ has exactly $2$ vertices from $X$ or exactly $2$ vertices from $Y$.  However, since $|X|$ and $|Y|$ are odd, $M_3$ does not contain a perfect $\U_r$-tiling.  

Suppose $2n/r\equiv 2\bmod{4}$.  We form  the standard multigraph $M_4$ on $n$ vertices similarly: Take two disjoint sets $X, Y$ of size $2n/r$.  Inside the sets $X, Y$ place all heavy edges and between $X$ and $Y$, place all light edges.  From $X\cup Y$ to the remaining vertices, place all heavy edges.  Now partition the remaining $(1-4/r)n$ vertices into sets of size $n/r$ or $2n/r$.  Between all such sets, place all heavy edges.  Inside the sets of size $2n/r$ place all light edges and inside the sets of size $n/r$ place no edges.  We have $\delta(M_4)=2(1-1/r)n-2$.  

Consider the standard multigraph $S_r$ with $r$ vertices in which there are all possible edges except that there is a vertex with precisely three light neighbours. If $r=4,5$ it is easy to check that
$S_r$ is not universal (however, it is universal for $r>5$). With this in mind, suppose that $r=4,5$.
 Then if $M_4$ contained a perfect $\U_r$-tiling, each copy of $\U_r$ would intersect the sets of size $n/r$ from $V(M_4)\setminus (X\cup Y)$ in exactly one vertex and the sets of size $2n/r$ in exactly two vertices, and furthermore every copy of $\U_r$ has exactly $4$ vertices from $X$ or exactly $4$ vertices from $Y$.  However, since $|X|$ and $|Y|$ are not divisible by $4$, $M_4$ does not contain a perfect $\U_r$-tiling. Note that if $r>5$ then $M_4$ does actually contain a perfect $\U_r$-tiling.\COMMENT{AT NEW: lots of changes}


Throughout the paper, instead of dealing with the set $\mathcal U_r$ itself, 
we will  mainly work with three subsets of $\mathcal U_r$: 
$\bar{\K}_r$, $\hat{\K}_r$ and $\K'_r$.
The elements of each of these three subsets are obtained from the complete standard multigraph on $r$ vertices by removing
the edges from a set of vertex-disjoint (light) paths $P_1, \dotsc, P_t$
where $|P_1| \ge |P_2| \ge \dotsc \ge |P_t|$.
The elements of $\bar{\K}_r$ are formed by removing the edges of $P_1, \dotsc, P_t$ from the complete standard multigraph on $r$ vertices
where we stipulate that $|P_i| \le 2$ for all $i \in [t]$, i.e.\ the elements of $\bar{\K}_r$ are formed
by removing a (light) matching from the complete standard multigraph on $r$ vertices.
The elements of $\hat{\K}_r$ are formed in the same way, but
we stipulate that $|P_1| \le 3$ and $|P_i| \le 2$ for all $i \ge 2$.
To form elements of $\K'_r$, we stipulate that either 
$|P_1| \le 4$ and $|P_i| \le 2$ for all $i \ge 2$, or $|P_1|, |P_2| \le 3$ and $|P_i| \le 2$ for all $i \ge 3$.

In the proof of Theorem~\ref{mainthm2} we (implicitly)
produce a perfect $\mathcal U_r$-tiling where \emph{most} of the tiles are elements from $\bar{\K}_r$  (in fact, we actually we produce a perfect $\K'_r$-tiling).\COMMENT{AT NEW: added}
The following example, however, demonstrates that we need the minimum degree to be greater than $2(1 - 1/r)n$ 
to guarantee a perfect $\bar{\K}_r$-tiling.
For any $k \in \mathbb N$ and $r \ge 3$, let $G$ be a standard multigraph containing $n := 2rk$ vertices constructed
in the following way.
Let $\{U_1, \dotsc, U_{r-1}\}$ be a partition of $V(G)$ such that
$|U_1| = 3k+1$, $|U_2| = 3k - 1$ and
$|U_i| = 2k$ for $i \in \{3, \dotsc, r-1\}$.
Place all possible light edges between $U_1$ and $U_2$,
and for all other pairs of distinct sets $U_i$ and $U_j$,
place all possible heavy edges.
Also, inside inside both $G[U_1]$ and $G[U_2]$
place all possible heavy edges.
The minimum degree of $G$ is $2(2rk - 2k) = 2(1 - 1/r)n$,
which is witnessed by any $v \in U_3 \cup \dotsm \cup U_{r-1}$.
Now suppose that $G$ contains $2k$ vertex-disjoint
elements from $\bar{\K}_r$.
Each one of these elements must have exactly
$1$ vertex in each of $U_3, \dotsc, U_{r-1}$
and exactly $3$ vertices in $G[U_1 \cup U_2]$;
however, $G[U_1 \cup U_2]$ does not contain a perfect $\bar{\K}_3$-tiling.
Note that when $r=3$, this corresponds to Wang's example in \cite{wang}.

As mentioned above, in the proof of Theorem~\ref{mainthm2} we actually produce a perfect $\K'_r$-tiling. Notice that for \emph{any} $r\geq 4$, $M_4$ does not contain a perfect $\K'_r$-tiling. So this multigraph is an extremal example for the perfect $\K'_r$-tiling problem for every $r\geq 4$.\COMMENT{AT NEW: added}

We now prove that $\bar{\K}_r\subseteq \hat{\K}_r\subseteq \K'_r\subseteq \U_r$ for $r\geq 4$.  While this follows from a very strong result of Havet and Thomass\'{e} \cite{ht00}, which states that every tournament $T$ on $n$ vertices contains every oriented path $P$ on $n$ vertices except when $P$ is the anti-directed path and $n \in \{3, 5, 7\}$, we prove it directly as we do not need the full strength of their result.

\begin{proposition}
For $r\geq 2$, $\bar{\K}_r\subseteq \U_r$ and for $r \ge 4$, $\bar{\K}_r\subseteq \hat{\K}_r\subseteq \K'_r\subseteq \U_r$.
\end{proposition}

\begin{proof} 
Let $r\geq 2$,  $T$ be a tournament on $r$ vertices, and let $\vec{K}$ be a legal orientation of $K\in \bar{\K}_r$, where $\vec{K}$ has exactly $t\leq \floor{r/2}$ light edges.  Form a bijection from $V(T)$ to $V(\vec{K})$ by choosing $t$ independent edges in $T$ and mapping their endpoints to the light edges of $\vec{K}$ with the correct orientation.  Then complete the bijection by mapping all other vertices of $T$ to $V(\vec{K})$ arbitrarily.  Since all other edges of $\vec{K}$ are double edges, we have $T\subseteq \vec{K}$.

Let $r\geq 4$,  $T$ be a tournament on $r$ vertices, and let $\vec{K}$ be a legal orientation of $K\in \K'_r$, where the light edges of $\vec{K}$ form $t$ vertex-disjoint paths $P_1,\dots, P_t$ with $|P_1|\geq\dots\geq |P_t|$ where either $|P_1| \le 4$ and $|P_i| \le 2$ for all $i \ge 2$, or $|P_1|, |P_2| \le 3$ and $|P_i| \le 2$ for all $i \ge 3$.  The statement follows from the following two facts which are straightforward to verify:
(1) Every tournament on $4$ vertices contains every orientation of a path on $4$ vertices; (2) Every tournament on $6$ vertices contains two vertex-disjoint transitive triangles.  We use this to first find an isomorphic copy of $P_1$ and $P_2$ (if applicable) in $T$, then we complete the embedding as in the first paragraph. 
\end{proof}


\section{Overview of the proof of Theorem~\ref{mainthm2}}
As with many proofs in the area, the proof of Theorem~\ref{mainthm2} divides into \emph{extremal} and \emph{non-extremal} cases. Roughly speaking, in the extremal case we consider those standard multigraphs that are `close' to the extremal examples
$M_1$, $M_2$, $M_3$ and $M_4$ \COMMENT{AT NEW: reworded} that were introduced after the statement of Theorem~\ref{mainthm2}. 
We deal with these extremal cases in one unified approach in Section~\ref{extremal}. 

Suppose that $G$ is as in Theorem~\ref{mainthm2}. Further, suppose that there is a `small' set $M \subseteq V(G)$ with the property that both $G[M]$ and $G[M\cup Q]$ contain
perfect $\mathcal U_r$-tilings for \emph{any} `very small' set $Q \subseteq V(G)$ where $|Q|\in r \mathbb N$. Then notice that, to find a perfect $\mathcal U_r$-tiling in $G$, it suffices to find an `almost' perfect $\mathcal U_r$-tiling
in $G':=G\setminus M$. Indeed, suppose that $G'$ contains a $\mathcal U_r$-tiling $\mathcal M_1$ covering all but a very small set of vertices $Q$. Then by definition of $M$, $G[M\cup Q]$  contains a perfect $\mathcal U_r$-tiling $\mathcal M_2$. Thus, $\mathcal M_1 \cup \mathcal M_2$ is a perfect $\mathcal U_r$-tiling in $G$, as desired. 

Roughly speaking, we refer to such a set $M$ as an `absorbing set' (see Section~\ref{secabs} for the precise definition of such a set). The `absorbing method' was first used in~\cite{rrs2} and has subsequently
been  applied to numerous embedding problems in extremal graph theory.

In general, a multigraph $G$ as in Theorem~\ref{mainthm2} may not contain an absorbing set. Indeed, consider the multigraph $G$ with  disjoint vertex classes $V_1,\dotsc,V_r$ each of size $n/r$ in which there are all possible heavy edges except that each $V_i$ is an independent set.\COMMENT{TM NEW: It didn't seem to me that the following statement was true with the graph as it was defined.  It is definitely true when the graph induced by $V(G) \setminus V_1$ does not contain a copy of $\mathcal U_r$.} Then if $Q$ is any set of $r$ vertices in $V_1$, there is no set $M \subseteq V(G)$ such that  both $G[M]$ and $G[M\cup Q]$ contain
perfect $\mathcal U_r$-tilings. Note that this multigraph is close to the extremal example $M_1$. It turns out that when $G$ is non-extremal, we can \emph{always} find an absorbing set $M$; we construct this set in Section~\ref{secabs}.

Thus, to complete the proof in the non-extremal case we must find an `almost' perfect $\mathcal U_r$-tiling in $G\setminus M$. Actually in Section~\ref{sec6} we prove a result that ensures \emph{any} multigraph $G$ as in Theorem~\ref{mainthm2} contains 
a  $\mathcal {\bar K}_r$-tiling covering almost all the vertices of $G$, see Corollary~\ref{almostthm_cor}. However, this does not quite guarantee a large enough $\mathcal {U}_r$-tiling in $G\setminus M$. Indeed, the leftover set $Q$ obtained by applying Corollary~\ref{almostthm_cor} to $G\setminus M$ will be slightly larger than the absorbing set $M$, and thus $M$ will not be able to absorb $Q$.

To overcome this we again have to use the property that our multigraph $G$ is non-extremal. Using the $\mathcal {\bar K}_r$-tiling obtained from Corollary~\ref{almostthm_cor} we build a significantly bigger $\mathcal {K}'_r$-tiling so that now the leftover set is very small compared to $M$. This is another delicate part of the proof and is dealt with in Section~\ref{sec7}.

In Section~\ref{sec6} we will apply a Tur\'an-type result for standard multigraphs; this is introduced in Section~\ref{secturan} (see Theorem~\ref{dituran}). We also introduce a multigraph regularity lemma in Section~\ref{secreg} and make use of this in Sections~\ref{sec6} and~\ref{sec7}.

\section{Notation}


For the rest of the paper, when we write multigraph, we mean standard multigraph.  Let $G$ be a multigraph. 
We write $e(G)$ for the total number of edges in $G$ and $e_2 (G)$ for the number of heavy edges in $G$.
Given a subset $X \subseteq V(G)$, we write $G[X]$ for the submultigraph of $G$ induced by $X$. We write $G \setminus X$ 
for the submultigraph of $G$ induced by $V (G) \setminus X$ and define $\overline{X}:=V(G)\setminus X$.

In a multigraph $G$, for $i=1,2$ let $N^i_G(v):=\{u: uv\in E(G) \text{ and } \mu(uv)=i\}$ and $d^i_G(v):=|N^i_G(v)|$. Let $N_G(v):=N^1_G(v)\cup N^2_G(v)$.  We define the \emph{degree $d_G(v)$ of $v$} to be the sum of the multiplicities of the edges incident with $v$, i.e. $d_G(v):=d^1_G(v)+2d^2_G(v)$.  
Note that $d_G(v) = |N_G(v)| + |N^2_G(v)|$.
Given a set $X \subseteq V(G)$ (or subgraph $X$ of $G$) we write $d_G(v,X)$ for 
the total number of edges in $G$ incident to $v$ whose other endpoint lies in $X$ (or $V(X)$). We define  $d^2 _G(v,X)$ similarly. Given disjoint $X,Y \subseteq V(G)$ (or subgraphs $X,Y$ of $G$) we write
$e_G(X,Y)$ for the total number of edges in $G$ with one endpoint in $X$ (or $V(X)$) and the other in $Y$ (or $V(Y)$);
We write $e_2 (X,Y)$ for the total number of heavy edges in $G$ with one endpoint in $X$ (or $V(X)$) and the other in $Y$ (or $V(Y)$), and let $E_2(X,Y)$ denote the set of all such edges.
 In all the aforementioned notation we omit the subscript $G$ if the multigraph is clear from the context. 

When $\U$ is a family of  multigraphs (digraphs) and $G$ is a multigraph (digraph) we write $\U \subseteq G$ to mean that some $U \in \U$ is a subgraph of $G$. If $U \in \U$ we say that $U$ is a \emph{copy of $\U$}.
If $\U$ is a family of  multigraphs and $G$ is a digraph we write $\U \subseteq G$ to mean that there is a legal orientation $\vec{U}$ of some $U \in \U$ such that 
$\vec{U}$ is a subdigraph of $G$.

Given a graph $G$ we let $G(t)$ denote the graph obtain from $G$ by replacing each vertex $x \in V(G)$ with a set $V_x$ of $t$ vertices so that, for all $x,y \in V(G)$: \COMMENT{TM NEW: slight change for clarity.}
\begin{itemize}
\item if $x \neq y$ then $V_x \cap V_y = \emptyset$;
\item If $xy \in E(G)$ then there are all possible edges in $G(t)$ between $V_x$ and  $V_y$;
\item If $xy \not \in E(G)$ then there are no edges in $G(t)$ between $V_x$ and  $V_y$.
\end{itemize}
Similarly, given a multigraph  $G$ we let $G(t)$ denote the multigraph obtain from $G$ by replacing each vertex $x \in V(G)$ with a set $V_x$ of $t$ vertices so that, for all $x,y \in V(G)$:
\begin{itemize}
\item if $x \neq y$ then $V_x \cap V_y = \emptyset$;
\item If $\mu(xy)=2$ in $G$ then there are all possible heavy edges in $G(t)$ between $V_x$ and  $V_y$;
\item If $\mu(xy)=1$ in $G$ then there are all possible light edges in $G(t)$ between $V_x$ and  $V_y$.
\item If $\mu(xy)=0$ in $G$ then there are no edges in $G(t)$ between $V_x$ and  $V_y$.
\end{itemize}

Given a set $X$ we write, for example $X+v$, $X-v$ and $X+v-w$ for $X \cup \{v\}$, $X\setminus \{v\}$ and $(X\setminus \{w\})\cup \{v\}$ respectively.
Similarly given multigraphs $T$, $G$ where $T\subseteq G$ and $v,w \in V(G)$, we write, for example, $T+v$, $T-v$, $T+v-w$ for the multigraphs $G[V(T)+v]$, $G[V(T)-v]$ and $G[V(T)+v-w]$ respectively.
We define, for example, $T-X$, $T+X$, $T-X+Y$ similarly where $X,Y \subseteq V(G)$.

Throughout the paper, we write $0<\alpha \ll \beta \ll \gamma$ to mean that we can choose the constants
$\alpha, \beta, \gamma$ from right to left. More
precisely, there are increasing functions $f$ and $g$ such that, given
$\gamma$, whenever we choose $\beta \leq f(\gamma)$ and $\alpha \leq g(\beta)$, all
calculations needed in our proof are valid. 
Hierarchies of other lengths are defined in the obvious way.

\section{Tur\'an-type results for universal multigraphs and digraphs}\label{secturan}
In this section we determine the density threshold that ensures a standard multigraph contains a universal graph, and therefore determine the threshold that forces
a digraph to contain any tournament of a given size.


Let $t_{r-1}(n)$ be the maximum number of edges in an $(r-1)$-partite graph on $n$ vertices and let $T_{r-1}(n)$ be the $(r-1)$-partite graph that realises this bound.  Note that when $n \ge r - 1$
\begin{equation}\label{difference}
t_{r-1}(n)-t_{r-1}(n-(r-1))=\frac{(r-1)(r-2)}{2}+(r-2)(n-(r-1)),
\end{equation}
and for any $n$
$$t_{r-1}(n) \le \left(1-\frac{1}{r-1}\right) \frac{n^2}{2},$$
with equality when $r-1$ divides $n$.

\begin{observation}\label{obs1}
Let $D_{r-1}(n)$ be the digraph obtained  by replacing every edge of $T_{r-1}(n)$ with two oppositely oriented directed edges and let $M_{r-1}(n)$ be the underlying multigraph of $D_{r-1}(n)$.  Then $D_{r-1}(n)$ contains no tournament on $r$ vertices and $M_{r-1}(n)$ contains no graph on $r$ vertices whose underlying graph is complete.
\end{observation}

Brown and Harary \cite{bh} proved that if a digraph $D$ on $n$ vertices contains more than $2t_{r-1}(n)$ edges, then $D$ contains every tournament on $r$ vertices. 
The following theorem strengthens their result by showing that $D$ contains subdigraph of $D$ on $r$ vertices which itself contains every tournament on $r$ vertices; in fact, we prove an even more general result about multigraphs.

\begin{theorem}\label{dituran}
Let $r\geq 2$ and let $G$ be a multigraph on $n$ vertices.  If $e(G)>2t_{r-1}(n)$, then $\Kb_{r} \subseteq G$.
\end{theorem}

\begin{proof}
The proof proceeds by double induction on $r$ and $n$. Clearly the result holds for $r=2$.  Let $r\geq 3$ and let $G$ be a standard multigraph on $n$ vertices such that $\Kb_{r}\not\subseteq G$.  If $n\leq r-1$, then $e(G)\leq 2t_{r-1}(n)$; so suppose $n\geq r$.  Either $e(G)\leq 2t_{r-2}(n)\leq 2t_{r-1}(n)$, or by induction, there exists a copy of $\Kb_{r-1}$ in $G$; let $H$ be a copy of $\Kb_{r-1}$
with the maximum number of edges.  If there exists $v\in V(G)\setminus V(H)$ such that $d(v, H)\geq  2(r-1)-1$, then we can either add $v$ to $H$ to make a copy of $\Kb_{r}$ or we can swap $v$ with a vertex in $H$ to increase the number of edges in $H$; either way, a contradiction.  So for all $v\in V(G)\setminus V(H)$ we have $d(v, H)\leq 2(r-1)-2=2(r-2)$.  Thus by \eqref{difference} and induction on $n$,
\begin{align*}
e(G)&=e(H)+e(G-H, H)+e(G-H)\\
&\leq (r-1)(r-2)+2(r-2)(n-(r-1))+2t_{r-1}(n-(r-1))\\
&=2\left(\frac{(r-1)(r-2)}{2}+(r-2)(n-(r-1))+t_{r-1}(n-(r-1))\right)\\
&=2t_{r-1}(n).
\end{align*}
\end{proof}
%

\begin{corollary}\label{cc2}
Let $G$ be a multigraph on $n$ vertices. If $\delta(G)>2(1-\frac{1}{r-1})n$ or $e(G)>(1-\frac{1}{r-1})n^2$, then $\Kb_{r}\subseteq G$.
\end{corollary}

First note that Theorem~\ref{dituran} and Corollary~\ref{cc2} immediately imply the analogous digraph versions.  Observation~\ref{obs1} shows that the density conditions in Theorem~\ref{dituran} and Corollary~\ref{cc2} to force a copy of $\Kb_r$ are best-possible; however, one may wonder if the same density conditions could force a multigraph $K$ whose complement contains a matching on at most $\floor{r/2}-1$ light edges.  The following observation shows that this is not the case.

\begin{observation}
Let $K$ be a multigraph on $r$ vertices such that the complement of $K$ is a matching with at most $\floor{r/2}-1$ light edges.  If $r$ is even, let $n \in r \mathbb N$ and if $r$ is odd, let $n\in (r+1)\mathbb{N}$.  For sufficiently large $n$, there exists a multigraph $G$ on $n$ vertices with (significantly) more than $2t_{r-1}(n)$ edges for which $K\not\subseteq G$.
\end{observation}

\begin{proof}
First suppose $r$ is even and $n=rk$.  Let $G$ be an $r/2$-partite multigraph with all parts of size $2k=2n/r$.  Inside each part put all possible light edges and between the parts put all possible heavy edges.  We have 
$$e(G)=n^2-n-\frac{r}{2}\binom{2n/r}{2}=\left( 1-\frac{1}{r}-\frac{1}{2n}\right )n^2,$$
which is much larger than $( 1-\frac{1}{r-1} )n^2\geq 2t_{r-1}(n)$.

Now suppose $r$ is odd and $n=(r+1)k$.  Let $G$ be an $(r+1)/2$-partite multigraph with all parts of size $2k=2n/(r+1)$.  Inside each part put all possible light edges and between the parts put all possible heavy edges.  We have $$e(G)=n^2-n-\frac{r+1}{2}\binom{2n/(r+1)}{2}=\left (1-\frac{1}{r+1}-\frac{1}{2n}\right )n^2,$$
which is much larger than $( 1-\frac{1}{r-1} )n^2\geq 2t_{r-1}(n)$.

Note that in each case  $G$ contains no copy of $K$.
\end{proof}

Finally, we address the issue of the structure of $\Kb_r$-free multigraphs with $2t_{r-1}(n)$ edges.  Let $\D_2^*(n)$ be the family of digraphs obtained by partitioning $n$ as $n=n_1+\dots+n_k$ such that $n_1,\dots, n_k$ are positive integers and at most one of the $n_i$s is odd, and taking $k$ disjoint copies, $D_1,\dots, D_k$ of $D_2(n_1), \dots, D_2(n_k)$, then adding all edges directed from $D_i$ to $D_j$ for all $1\leq i<j\leq k$.  In particular, note that $D_2(n)\in \D_2^*(n)$.  Brown and Harary~\cite{bh} proved that if $T\in \mathcal T_r$  and $D$ is a $T$-free digraph on $n$ vertices with $2t_{r-1}(n)$ edges, then $D\cong D_{r-1}(n)$ unless $T=C_3$ in which case $D\in \D_2^*(n)$.
The following observation shows that in our case, there is a whole family of tightness examples.  Let $\M_2^*(n)$ be the family of multigraphs underlying the digraphs in $\D_2^*(n)$.

\begin{observation} Given $r\geq 4$,
let $n\in r\mathbb{N}$ and let $\M^*_{r-1}(n)$ be the family of multigraphs on $n$ vertices which can be obtained from $M_{r-1}(n)$ by the following process.  Take disjoint pairs of colour classes and replace each such pair with a copy of $M\in \M^*_2(2n/r)$, leaving all other edges between the sets as they were.
Then every $M\in \M^*_{r-1}(n)$ does not contain $\mathcal {\bar K}_r$.
\end{observation}

\begin{proof}
Let $M\in \M^*_{r-1}(n)$, let $X_1,\dots, X_s$ be the colour classes from $M$ which were not modified, and let $Y_1, \dots, Y_t$ be the sets from $M$ which appeared as a result of merging two of the original colour classes.  We have $r-1=s+2t$ and thus any copy of $\Kb_r$ must contain at least $2$ vertices from some $X_i$, which is clearly not possible, or at least $3$ vertices from some $Y_j$, which would imply that $Y_j$ contains a copy of $\Kb_3$, which is not the case.  Thus $\Kb_r \notin M$.
\end{proof}

It would be interesting to determine whether every $\Kb_r$-free multigraph on $2t_{r-1}(n)$ edges is a member of $\M^*_{r-1}(n)$, and more generally, whether every $\Kb_r$-free multigraph on $2t_{r-1}(n)-o( n^2)$\COMMENT{AT NEW: changed $\eps$ to $o$} edges is sufficiently ``close'' (in edit-distance) to some member of  $\M^*_{r-1}(n)$. \COMMENT{LD: I didn't really think about either question.  Should we mention these stabilty questions for the Tur\'an problem? AT NEW: happy to leave as it is}

\section{A Regularity Lemma for Standard Multigraphs}\label{secreg}
In the proof of Theorems~\ref{mainthm2} and~\ref{almostthm_stability} we will apply a version of 
Szemer\'edi's regularity lemma~\cite{reglem} for multigraphs. Before we  state it we need some more definitions. 
The \emph{density} of a bipartite graph $G=(A,B)$ with vertex classes~$A$ and~$B$ is defined to be 
$$d_G (A,B):=\frac{e_G(A,B)}{|A||B|}.$$ 
We will write $d(A,B)$ if this is unambiguous. Given any $\varepsilon >0$ we say that~$G$
is {\it $\varepsilon$-regular} if for all $X\subseteq A$ and $Y \subseteq B$ with $|X|>\varepsilon 
|A|$ and $|Y|> \varepsilon |B|$ we have that $|d(X,Y)-d(A,B)|<\varepsilon$. 

Given disjoint vertex sets~$A$ and~$B$ in a graph~$G$, we write $(A,B)$
for the induced bipartite subgraph of~$G$ whose vertex classes are~$A$ and~$B$.
If $G$ is a  multigraph and $A, B \subseteq V(G)$ are disjoint, then we write $(A,B)^i _G$ for the bipartite \emph{graph} with vertex classes $A$ and $B$
where $a \in A$ and $b \in B$ are adjacent in $(A,B)^i _{G}$ precisely if $\mu (ab)=i$ in $G$.

The next well-known observation (see~\cite{ko} for example) states that a large subgraph of a  regular pair is also regular.
\begin{lemma}\label{slice}
Let $0< \eps < \alpha$ and $\eps ':= \max \{ \eps /\alpha , 2\eps \}$. Let $(A,B)$ be an $\eps$-regular pair of density $d$. Suppose $A' \subseteq A$ and $B' \subseteq B$ where $|A'|\geq \alpha |A|$ and 
$|B'|\geq \alpha |B|$. Then $(A',B')$ is an $\eps '$-regular pair with density $d'$ where $|d'-d|<\eps$.
\end{lemma}

The following  result will be applied in the proof of Theorem~\ref{almostthm}. It is (for example) a special case of Corollary 2.3 in~\cite{alony}.
\begin{lemma}\label{red}
Let $\eps, d>0$ and $m,r,t \in \mathbb N$ such that $0<1/m \ll \eps \ll d \ll 1/r$ and $t\leq r$. Let $H$ be a graph obtained from $K_r$ by replacing every vertex of $K_r$ with $m$ vertices and replacing each edge of $K_r$ with an $\eps ^2$-regular pair of density at least $d$. Then $H$
contains a $K _t$-tiling covering all but at most $\eps m r$ vertices.
\end{lemma}

We apply the following version of the regularity lemma, which is an immediate corollary of a $2$-coloured regularity lemma from~\cite{blssw} (Theorem~2.4).
This result in turn  is easy to derive from the many-colour regularity lemma presented in~\cite{ks} (Theorem~1.18).

\begin{lemma}[Degree form of Multigraph Regularity Lemma]\label{2colordegreeform}
For any $\ep>0$ and  $M'\in \mathbb N$, there exists $M=M(\ep, M')$ such that the following holds.  
Let $G$ be a standard multigraph on $n$ vertices and let $0\leq d\leq 1$.  Then there exists a partition $\{V_0, V_1,\dots, V_k\}$ of $V(G)$ with $M' \le k \le M$ and a spanning subgraph $G'$ of $G$ with the following properties:
\begin{enumerate}
\item $|V_0| \le \ep n$;
\item all clusters $V_i$, $i\in [k]$, are of the same size $\frac{1-\ep}{M}n\leq \frac{n-|V_0|}{k}= |V_1|\le \frac{n}{M'}$;
\item $d_{G'}(v)>d_G(v)-(4d+2\ep)n$ for all $v\in V(G')$;
\item $e(G'[V_i]) = 0$ for all $i \in [k]$;
\item for all $1 \le i < j \le k$ and $c\in [2]$, the pair $(V_i,V_j)^c _{G'}$ is $\ep$-regular  with density either 0 or at least $d$.
\end{enumerate}
\end{lemma}
We call $V_1, \dots, V_k$ \emph{clusters}, $V_0$ the \emph{exceptional set} and  $G'$  the \emph{pure multigraph}.  Given a multigraph  $G$, and parameters $\ep, d, M'$, we define the \emph{reduced multigraph $\Gamma$} as follows:  Let $\{V_0,V_1,\dots, V_k\}$ be the partition and $G'$ be the subgraph of $G$ obtained from an application of Lemma~\ref{2colordegreeform} with 
parameters $\ep, d, M'$.  We let $V(\Gamma) = \{V_1,\dots, V_k\}$ and (i) if $(V_i,V_j)^2 _{G'}$ has density at least $d$ we place a heavy edge between $V_i$ and $V_j$ in $\Gamma$;
(ii) if $(V_i,V_j)^2 _{G'}$ has density  $0$ and $(V_i,V_j)^1 _{G'}$ has density at least $d$ we place a light edge between $V_i$ and $V_j$ in $\Gamma$;
(iii) otherwise $V_i$ and $V_j$ are not adjacent in $\Gamma$. 

The next result implies that the minimum degree of a multigraph is almost inherited by its reduced multigraph.

\begin{lemma}\label{inherit}
Let $\ep>0$, $d\in [0,1]$, $M',n \in \mathbb N$ and  let $G$ be a multigraph on $n$ vertices. Let $G'$ be the pure multigraph and $\Gamma$ be the reduced multigraph obtained by applying Lemma \ref{2colordegreeform} to $G$ with parameters $\ep$, $d$ and $M'$.  Then $\delta(\Gamma)\geq (\delta(G)/n-(8d+6\ep))|\Gamma|$.
\end{lemma}

\begin{proof}
Note that for all $V_i\in V(\Gamma)$,
\begin{equation*}
d_{\Gamma}(V_i)=d^1_{\Gamma}(V_i)+2d^2_{\Gamma}(V_i)=|N_{\Gamma}(V_i)|+d^2_{\Gamma}(V_i).
\end{equation*}
Let $v\in V_i$.  Notice that $N_{G'}(v)$ intersects at least $(|N_{G'}(v)|-|V_0|)/|V_1|$ 
clusters and thus by Lemma \ref{2colordegreeform}, 
$$|N_{\Gamma}(V_i)|\geq (|N_{G'}(v)|-|V_0|)/|V_1|\geq (d^1_G(v)+d^2_G(v)-(4d+3\ep)n)/|V_1|.$$   
Also note that $N^2_{G'}(v)$ intersects at least $(|N^2_{G'}(v)|-|V_0|)/|V_1|$ clusters and thus by Lemma \ref{2colordegreeform},
$$d^2_{\Gamma}(V_i)\geq (|N^2_{G'}(v)|-|V_0|)/|V_1|\geq (d^2_G(v)-(4d+3\ep)n)/|V_1|.$$  
Altogether this gives 
\begin{align*}
d_{\Gamma}(V_i)=|N_{\Gamma}(V_i)|+d^2_{\Gamma}(V_i)&\geq (d^1_G(v)+d^2_G(v)-(4d+3\ep)n)/|V_1|+(d^2_G(v)-(4d+3\ep)n)/|V_1|\\
&=(d^1_G(v)+2d^2_G(v)-(8d+6\ep)n)/|V_1|\\
&\geq (\delta(G)/n-(8d+6\ep))|\Gamma|.
\end{align*}
Therefore $\delta(\Gamma)\geq (\delta(G)/n-(8d+6\ep))|\Gamma|$, as claimed.

\end{proof}

\section{Almost tiling multigraphs with $\bar{\mathcal K}_r$}\label{sec6}
In order to prove Theorem~\ref{mainthm2}, we will apply (a corollary of) the following result. Roughly speaking, it states that every standard multigraph with minimum degree slightly greater than that in Theorem~\ref{mainthm2} contains an almost 
perfect $\bar{\mathcal K}_r$-tiling.

\begin{theorem}\label{almostthm}
Let $n,r \in \mathbb N$ where $r \geq 2$ and $\eta >0$ such that $0<1/n \ll \eta \ll 1/r$. Suppose that $G$ is a standard multigraph on $n$ vertices such that 
$$\delta (G) \geq 2(1-1/r+\eta )n.$$
Then $G$ contains a $\Kb_r$-tiling covering all but at most $\eta n$ vertices.
\end{theorem}
The next result is the key tool in the proof of Theorem~\ref{almostthm}.

\begin{lemma}\label{expand}
Let $\eta ,\gamma >0$ and $n,r \geq 2$ be integers such that $0 < 1/n \ll \gamma \ll \eta \ll 1/r$. Let $G$ be a standard multigraph on $n$ vertices so that
\begin{align}\label{min} 
\delta (G) \geq 2(1-1/r+\eta)n.
\end{align}
Further, suppose that the largest $\Kb_r$-tiling in $G$ covers precisely $n' \leq (1- \eta )n$ vertices. Then there exists a $(\Kb_r \cup \bar{\mathcal K}_{r+1})$-tiling in $G$ that covers at least $n'+ \gamma n$ vertices.
\end{lemma}
\proof
Certainly Theorem~\ref{dituran} and (\ref{min}) imply that $n' \geq \eta n$. Let $\mathcal M$ denote a $\Kb_r$-tiling in $G$ containing precisely $n'$ vertices so that the total number of edges in $\mathcal M$
is maximised. Set $n'' :=n-n'$ and $G'':=G\setminus V(\mathcal M)$.
\begin{claim}\label{claimy}
There are at least $\gamma n$ vertices $x \in V(G'')$ such that $d _G (x, V(\mathcal M))\geq 2(1-1/r)n' + 2\gamma n$.
\end{claim}
Suppose for a contradiction the claim is false. Then by (\ref{min}), at least $n'' -\gamma n$ vertices $y \in V(G'')$ are such that 
$d_{G''} (y) \geq 2(1-1/r+\eta )n -2(1-1/r)n'- 2\gamma n \geq 2(1-1/r+\eta /2)n''$.  Thus by Theorem~\ref{dituran}, $G''$ contains a copy of some $U \in \Kb_r$.
But then together with $\mathcal M$, this forms a $\bar{\mathcal K}_r$-tiling on $n'+r$ vertices in $G$, a contradiction to the maximality of $\mathcal M$. This proves the claim.

\medskip

Given any $x \in V(G'')$ such that $d_G (x, V(\mathcal M)) \geq 2(1-1/r)n'+ 2\gamma n$, there are at least $\gamma n$ elements $U$ in $\mathcal M$ so that $d_G(x, U)\geq 2r-1$.
If $d_G(x, U)=2r$, then since $U \in \Kb_r$, $V(U) \cup \{x\}$ spans a copy of an element of $\bar{\mathcal K}_{r+1}$ in $G$. Otherwise there is precisely one vertex $y \in V(U)$ such that $xy$ is a light edge.

Suppose that for some $z \in V(U)\setminus \{y\}$, $zy$ is a light edge in $G$. Then since there are all possible edges between $x$ and $V(U)\setminus \{y\}$, $(V(U)\setminus \{y\})\cup \{x\}$ spans a copy of some $U'\in \Kb_r$ such that
$|E(U')|=|E(U)|+1$ \COMMENT{TM NEW: added missing cardinality notation.}. This is a contradiction to the choice of $\mathcal M$.  Thus for all $z\in V(U)\setminus \{y\}$, $yz$ is a heavy edge. This implies that $V(U)\cup \{x\}$ spans a copy of some $U' \in \bar{\mathcal K}_{r+1}$.

 Claim~\ref{claimy} implies there are at least $\gamma n$ vertices $x \in V(G'')$ such that $d _G (x, V(\mathcal M))\geq 2(1-1/r)n' +2\gamma n$. So for at least $\gamma n$ such vertices $x$, we can pair them
off with distinct elements $U$ of $\mathcal M$ so that $V(U)\cup \{x\}$ spans a copy of an element of  $\bar{\mathcal K}_{r+1}$. This therefore implies that there exists a $(\Kb_r \cup \bar{\mathcal K}_{r+1})$-tiling in $G$ that covers at least $n'+ \gamma n$ vertices, as desired.
\endproof

The next simple observation will be used in the proof of Theorem~\ref{almostthm} to convert a $(\bar{\K}_r \cup \bar{\mathcal K}_{r+1})$-tiling in the reduced multigraph $\Gamma$ of $G$ into a
$\bar{\K}_r$-tiling in the blow-up $\Gamma(r)$ of $\Gamma$.
\begin{fact}\label{fact2} Suppose that $r ,t\in \mathbb N$ such that $r$ divides $t$. If $U \in (\bar{\K}_r \cup \bar{\mathcal K}_{r+1})$ then $U(t)$ contains a perfect $\bar{\K}_r$-tiling.
\end{fact}

We are now ready to prove Theorem~\ref{almostthm}. We will repeatedly
apply Lemma~\ref{expand} and Fact~\ref{fact2} to obtain an almost
perfect $\bar{\K}_r$-tiling in a blow-up of the reduced multigraph of $G$.
Applying Lemma~\ref{red} will then yield an almost
perfect $\bar{\K}_r$-tiling in $G$.
 Arguments of a similar nature were applied in~\cite{ko, hlad, triangle}.

{\noindent \bf Proof of Theorem~\ref{almostthm}.}
Define additional constants $\eps, d, \gamma$ and $M' \in \mathbb N$ so that $0<1/n \ll 1/M' \ll \eps \ll d \ll \gamma \ll \eta \ll 1/r$. Set $z:= \lceil 1/\gamma \rceil $. Apply Lemma~\ref{2colordegreeform} with parameters $\eps, d$ and $M'$ to $G$ 
to obtain clusters $V_1, \dots , V_k$, an exceptional set $V_0$ and a pure multigraph $G'$.
Set $m:=|V_1|=\dots =|V_k|$.
 Let $\Gamma$ be the reduced multigraph of $G$ with parameters $\eps, d$ and $M'$. Lemma~\ref{inherit} implies that 
\begin{align}\label{ds*} 
\delta (\Gamma) \geq 2(1-1/r+\eta /2)k.
\end{align}

\begin{claim}\label{blowclaim}
 $\Gamma':=\Gamma(r^z)$ contains a $\bar{\K}_r$-tiling covering at least $(1-\eta /2)kr^z=(1-\eta /2)|\Gamma'|$ vertices. 
\end{claim}
If $\Gamma$ contains a $\bar{\K}_r$-tiling covering at least $(1-\eta /2)k$ vertices then Fact~\ref{fact2} implies that Claim~\ref{blowclaim} holds. So suppose that the largest $\bar{\K}_r$-tiling in $\Gamma$ covers precisely
$\ell \leq (1- \eta /2)k$ vertices. Then by Lemma~\ref{expand}, $\Gamma$ contains a $(\bar{\K}_r \cup \bar{\K}_{r+1})$-tiling that covers at least $\ell+ \gamma k$ vertices. Thus, by Fact~\ref{fact2}, $\Gamma(r)$ contains a $\bar{\K}_r$-tiling
covering at least $(\ell+\gamma k)r$ vertices. (So at least a $\gamma$-proportion of the vertices in $\Gamma(r)$ are covered.) Further, by
definition of $\Gamma (r)$ and (\ref{ds*}),
\begin{align*} 
\delta (\Gamma(r)) \geq 2(1-1/r+\eta /2)kr.
\end{align*}
If $\Gamma(r)$ contains a $\bar{\K}_r$-tiling covering at least $(1-\eta /2)kr$ vertices then again Fact~\ref{fact2} implies that the claim holds.
So suppose that the largest $\bar{\K}_r$-tiling in $\Gamma(r)$ covers precisely
$\ell' \leq (1- \eta /2)kr$ vertices. Recall that $\ell' \geq (\ell+\gamma k)r$. By Lemma~\ref{expand}, $\Gamma(r)$ contains a $(\bar{\K}_r \cup \bar{\K}_{r+1})$-tiling that covers at least $\ell' +\gamma kr \geq (\ell+2\gamma k)r$ vertices.
Thus, by Fact~\ref{fact2}, $\Gamma(r^2)$ contains a $\bar{\K}_r$-tiling
covering at least $(\ell+2\gamma k)r^2$ vertices. (So at least a $2\gamma$-proportion of the vertices in $\Gamma(r^2)$ are covered.) Repeating this argument at most $z$ times we see that the claim holds.

\medskip

For each $1 \leq i \leq k$, partition $V_i$ into classes $V^* _i, V_{i,1}, \dots , V_{i, r^z}$ where $m':=|V_{i,j}|= \lfloor m/r^z \rfloor \geq m/(2r^z)$ for all $1 \leq j \leq r^z$.
Since $mk \geq (1-\eps)n$ by Lemma~\ref{2colordegreeform},
\begin{align}\label{m'}
m'|\Gamma '| = \big \lfloor {m}/{r^z} \big \rfloor kr^z \geq mk-kr^z \geq (1-2\eps)n.
\end{align}

Let $c\in [2]$.
Lemma~\ref{slice} implies that if $(V_{i_1}, V_{i_2})^{c}_{G'}$ is $\eps$-regular with density at least $d$ then $(V_{i_1, j_1}, V_{i_2,j_2})^{c}_{G'}$ is $2\eps r^z$-regular with density at least $d-\eps \geq d/2$
(for all $1\leq j_1,j_2 \leq r^z$). In particular, we can label the vertex set of $\Gamma '$ so that $V(\Gamma ')=\{V_{i,j} : 1 \leq i \leq k , \ 1 \leq j \leq r^z \}$ where, for $c \in [2]$,
$\mu (V_{i_1, j_1} V_{i_2,j_2})=c$ in $\Gamma'$ implies that $(V_{i_1, j_1}, V_{i_2,j_2})^c _{G'}$ is $2\eps r^z$-regular with density at least $d/2$.

By Claim~\ref{blowclaim}, $\Gamma '$ has a $\bar{\K}_r$-tiling $\mathcal M$ that contains at least $(1-\eta /2)|\Gamma'|$ vertices. Consider any element $U$ in $\mathcal M$ and let
$V(U)=\{ V_{i_1, j_1}, V_{i_2,j_2}, \dots , V_{i_r, j_r} \}$. Set $V'$ to be the union of $ V_{i_1, j_1}, V_{i_2,j_2}, \dots , V_{i_r, j_r}$. Note that $0<1/m' \ll 2 \eps r^z \ll d/2 \ll \gamma \ll 1/r$.
Thus, Lemma~\ref{red} implies that $G'[V']$ contains a $\bar{\K}_r$-tiling covering all but at most $\sqrt{2 \eps r ^z} m' r \leq \gamma m' r$ vertices. 
(Here we are using that a heavy edge in $\Gamma '$ corresponds to a
$2\eps r^z$-regular pair in $G'$ consisting only of heavy edges, and a light edge in $\Gamma '$ corresponds to a
$2\eps r^z$-regular pair in $G'$ consisting only of light edges.)
By considering each element in $\mathcal M$
we conclude that
$G'\subseteq G$ contains a $\bar{\K}_r$-tiling covering at least
$$ (1-\gamma )m'r \times (1- \eta /2)|\Gamma'|/r \stackrel{(\ref{m'})}{\geq} (1-\gamma)(1-\eta /2) (1-2\eps )n \geq (1-\eta )n $$
vertices, as desired.
\endproof

The following result is a simple consequence of Theorem~\ref{almostthm}.
\begin{corollary}\label{almostthm_cor}
Let $r\geq 2$ and $0 < 1/n \ll \eta \ll 1/r$.  Suppose that $G$ is a multigraph on $n$ vertices such that $$\delta(G) \geq 2(1-1/r-\eta)n.$$ 
Then $G$ contains a $\bar{\K}_r$-tiling covering all but at most $4r^2\eta n$ vertices.
\end{corollary}

\begin{proof}
 Add $n': = \ceiling{2 \eta n/(1/r - \eta)}$ vertices to $G$ which 
 send out heavy edges to all other vertices (including each other).
Call the resulting multigraph $G^*$. 
  Since
  $$\delta(G^*)=\delta(G) + 2 n' = \delta(G) + 2(1 - 1/r  + \eta)n' + 2(1/r - \eta)n' \ge 2(1 - 1/r + \eta)(n + n'),$$ 
  we can apply Theorem~\ref{almostthm} to $G^*$ to obtain a 
  $\bar{\K}_r$-tiling in $G^*$ covering all but at most $\eta(n+n')$ vertices. Removing all those tiles that contain vertices from $V(G^*)\setminus V(G)$, we obtain a $\bar{\K}_r$-tiling in $G$ that covers all but at most 
   $\eta(n + n') + (r-1)n' \le 4r^2 \eta n$ of the vertices of $G$, as desired.
\end{proof}


\section{Almost perfect tilings in the non-extremal case}\label{sec7}
Suppose that, in the proof of Theorem~\ref{mainthm2}
we have found a small absorbing set $M$. 
 Ideally, we would next like to apply  Corollary~\ref{almostthm_cor} to conclude 
that $G\setminus V(M)$ contains an almost perfect $\mathcal U_r$-tiling $\mathcal M$, and 
then use $M$ to cover the remaining vertices, thereby obtaining a perfect
$\mathcal U_r$-tiling in $G$. However, to achieve this we would require that the
set of vertices uncovered by $\mathcal M$ is \emph{much smaller} than the size of the absorbing set $M$. Corollary~\ref{almostthm_cor} does not guarantee this though. Indeed, this is because the size of the set of uncovered vertices in  Corollary~\ref{almostthm_cor} is \emph{large} compared to the parameter $\eta$. 
Worst still, it is easy to see that the conclusion of  Corollary~\ref{almostthm_cor} is false if we replace $4r^2 \eta n$ with a term significantly smaller than $\eta n$.  

Therefore, instead we will show that the conclusion of Corollary~\ref{almostthm_cor}
can be strengthened in the desired way if our multigraph $G$ is \emph{far} from extremal.
(This strengthening will be at the cost of no longer guaranteeing an
almost perfect $\bar{\K}_r$-tiling, but rather an almost perfect ${\K}'_r$-tiling.)
 This will ensure that we can then use the above approach in the non-extremal case (we then have to deal with the extremal case separately).  \COMMENT{LD: I don't feel too strongly one way or the other, but these two paragraphs might not be needed in light of the Section 2 overview.  Some of the same discussion appears in both places.
AT NEW: I would definitely like to keep it as it is. I think it is okay to have some overlap; in general I think that can help re-enforce things.}

To precisely describe the multigraphs that are far from extremal,
we use the following definition.
\COMMENT{TM NEW: added definition environment for extremal.}
\begin{definition}
  \label{def:extremal}
 Given $\gamma >0$ and $r \in \mathbb N$,
  we say that a multigraph $G$ on $n$ vertices is \emph{$(1/r, \gamma)$-extremal} if 
\begin{enumerate}
\item there exists $S\subseteq V(G)$ such that $||S|-n/r|< \gamma n$ and $e(G[S])<\gamma n^2$, or
\item there exists $S\subseteq V(G)$ such that $||S|-2n/r|<\gamma n$ and $e_2(G[S])<\gamma n^2$.
\end{enumerate}
\end{definition}

The aim of this section is to prove the following result.

\begin{theorem}\label{almostthm_stability}
  Let $n,r \in \mathbb N$ where $r \geq 2$ and $\alpha, \eta, \gamma >0$ such that $0<1/n \ll \alpha \ll \eta \ll \gamma \ll 1/r$ and let $G$ be a multigraph on $n$ vertices.  If $G$ is not $(1/r, \gamma)$-extremal and
$$\delta (G) \geq 2(1-1/r-\eta )n,$$
then $G$ contains a $\K_r'$-tiling covering all but at most $\alpha n$ vertices.
\end{theorem}


  The proof of Theorem~\ref{almostthm_stability} makes use of Corollary~\ref{almostthm_cor}. The next result will be used to convert an almost perfect tiling of a multigraph  with universal graphs into a perfect tiling.
	
	\begin{lemma}\label{lem_oneleftover}
  Let $n,r \in \mathbb N$ where $r \geq 2$ and $\tau , \gamma '>0$ such that $0<1/n \ll \tau \ll \gamma' \ll 1/r$ and let $G$ be a multigraph on $n$ vertices.  
  If $G$ is not $(1/r, \gamma')$-extremal,
$\delta (G) \geq 2(1-1/r-\tau)n,$ and there exists a $\bar{\K}_r$-tiling covering all but  one vertex, then $G$ contains a perfect $(\K_r'\cup \K_{r+1}')$-tiling in which all but at most three  of the tiles are copies of $\bar{\K}_r$.
\end{lemma}

\begin{proof}
  Let $\T=\{T_1, \dots, T_m\}$ be the $\bar{\K}_r$-tiling in $G$ and let $v^*$ be the leftover vertex.  If there exists $T\in \T$ such that $d(v^*, T)\geq 2r-1$, then we obtain a perfect $(\K_r'\cup \K_{r+1}')$-tiling in which all but one of the tiles are copies of $\bar{\K}_r$.  So suppose that this is not the case. Then for all but at most 
    \begin{equation}
      \label{eq:2r-2}
      (2r - 2)|\T| - \delta(G) = 2(1 - 1/r)(n-1) - \delta(G) \le 3 \tau n 
    \end{equation}
  of the $T\in \T$, $d(v^*, T)=2r-2$. 

If there exists $T\in \T$ such that $d(v^*, T)= 2r-2$, then we could move $v^*$ into $T$ to create a copy of $\K_{r+1}'$ unless:
\begin{enumerate}
\item[($\alpha$)]  there exists $u\in V(T)$ such that $\mu(v^*u)=0$ in $G$ or;

\item[($\beta$)] $u_1, u_2\in V(T)$ are light neighbours of $v^*$ in $G$ and either

\begin{enumerate}
\item\label{2a} $u_1u_2$ is a light edge or;

\item\label{2b} $u_1$ and $u_2$ are incident with distinct light edges in $T$.

\end{enumerate}
\end{enumerate}
So we may suppose that one of ($\alpha$) and ($\beta$) holds whenever $d(v^*, T)= 2r-2$.

Let 
$$B^*_i:=\{u\in T\in \T: d(v^*, T)=2r-2 \text{ and } \mu(v^*u)=i\}.$$
Set $B^*:=B^*_0\cup B^*_1$. 
Let $\T^*$ to be the set of tiles in $\T$ which contain a vertex from $B^*$. 
Call an edge of $G$ \emph{useful} if it either has both endpoints in $B_0^*$, or it has one endpoint in $B_0^*$ and the other in $B_1^*$, or it is a heavy edge with both endpoints in $B_1^*$.  



Given distinct $T,T' \in \mathcal T^*$, we say that the ordered pair $(T,T')$ is \emph{bad}\COMMENT{LD: The definition of bad slightly changed to $d(b,T')<2r-2$ instead of $d(b,T')\neq 2r-2$; I think this makes sense regardless of the other changes. AT NEW: I do prefer how it originally was, particularly in light of the reordering: It is then explicit that the hypothesis of Claim 7.4 is satisfied by the `useful good pair' $T,T'$. TM NEW: Changed back to $d(x, T') \neq 2r - 2$ and $d(b, T') \neq 2r - 2$.} if
\begin{enumerate}
\item there exists $b\in B^*\cap T$ such that $d(b, T') \neq 2r-2$ or failing this;

\item there exists some  $b'\in B^* \cap T'$ such that $d(b', T)=2r-2$, and some $x \in T$ such that $b'x$ is not a heavy edge and $d(x, T') \neq 2r-2 $.
\end{enumerate}
If neither $(T,T')$ nor $(T',T)$ are bad, then we say that $\{T, T'\}$ is \emph{good}.  

First we will show that non-extremality guarantees that there is a useful edge $bb'$ between a good pair $\{T,T'\}$.  Then we will show how to use such a configuration to get the desired tiling which uses $v^*$.  

Define an auxiliary digraph $\cD$ with vertex set $\T^*$ where there is an edge from $T$ to $T'$ if $(T,T')$ is bad.

Note that if there exists a $b \in B^*$ where $b\in T\in \T^*$ and a $T'\in \T \setminus \{T\}$ such that $d(b, T')\geq 2r-1$, then we obtain our desired perfect $(\K_r'\cup \K_{r+1}')$-tiling by moving $b$ to $T'$ to create a $\hat{\K}_{r+1}$ and then moving $v^*$ to $T$ to create a $\hat{\K}_r$ ($b$ was a problem vertex for $v^*$; moving it out of $T$ means that we can move $v^*$ in).  So we may assume that such a $b$ does not exist.
Therefore, for all $b$ in $B^*$, by a computation similar to \eqref{eq:2r-2}, we have that $d(b, T')=2r-2$ for all but at most $3 \tau n$ of the $T'\in \T$.
In particular, for a fixed $T \in \mathcal T^*$ there are at most $6 \tau n$ $T' \in \mathcal T^*$ such that $(T,T')$ satisfies (i) in the definition of bad.

Given distinct $T,T' \in \mathcal T^*$,
suppose there exists $b'\in B^* \cap T'$ such that $d(b', T)=2r-2$, and some $x \in T$ such that $b'x$ is not a heavy edge and $d(x, T')\geq 2r-1$. 
By the previous paragraph, $x \not \in B^*$.
So $T\cup T' \cup \{v^*\}$ spans two disjoint copies of $\hat{K}_r $ and $K'_{r+1}$.
(The vertex set of the former tile is $T-x+b'$, the latter $T'-b'+x+v^*$.)  In particular, we obtain our desired perfect $(\K_r'\cup \K_{r+1}')$-tiling. So we may assume that such an $x$ does not exist.

Given distinct $T,T' \in \mathcal T^*$,
suppose there exists $b'\in B^* \cap T'$ such that $d(b', T)=2r-2$, and some $x \in T$ such that $b'x$ is not a heavy edge. 
  Further, suppose  there exists $T''\in \T \setminus \{T,T'\}$ such that $d(x, T'')\geq 2r-1$.  We can move $x$ to $T''$ to create a $\hat{\K}_{r+1}$, move $b'$ to $T$ to create a $\hat{\K}_r$ and move $v^*$ to $T'$ to create a $\hat{\K}_r$.  
	In particular, we obtain our desired perfect $(\K_r'\cup \K_{r+1}')$-tiling. Thus, we may assume that this is not the case.
	
	Fix $T \in \mathcal T^*$. Suppose there are at least $3r \tau n$ $T'\in \mathcal T^*$ such that $(T,T')$ satisfies (ii) in the definition of bad. Then there exists some vertex $w\in T$ that plays the role of $x$ in (ii) for at
	least $3\tau n$ such $T'$. But then the previous two paragraphs imply that
	$$d_G (w) \leq 3\tau n (2r-3)+(|\mathcal T|-3\tau n)(2r-2)+2 < \delta (G),$$
	a contradiction.

Altogether this implies that $\cD$ has maximum out-degree at most $6r \tau n$ and so $e(\cD)\leq 6r\tau n^2$.

We will now show that there are more than $6 r\tau  n^2$ (unordered) pairs $\{T,T'\}$ where $T,T' \in \mathcal T^*$ and so that there is a useful edge in $G$ with one endpoint in $T$ and the other in $T'$.  Then for at least one such $\{T,T'\}$ we have that neither $(T,T')$ nor $(T',T)$ is a bad pair. 

By the non-extremality of $G$, if $|B^*_0| \ge (1 - \gamma')\frac{n}{r}$ we have at least $\gamma' n^2\gg 6 r\tau  n^2$ useful edges in $G[B^*_0]$.
If $|B^*|\geq (1-\gamma')\frac{2n}{r}$, then we have at least $\gamma' n^2 \gg 6 r\tau  n^2$ heavy edges in $G[B^*]$, all of which are useful (and at most $n$ of these edges go between vertices in the same tile from $\mathcal T^*$).
Note that there are at most $4$ useful edges between any $T$ and $T'$ from $\mathcal T^*$.  
So we can assume that both $|B^*_0| < (1 - \gamma')\frac{n}{r}$ and $|B^*| < (1 - \gamma')\frac{2n}{r}$.
With the fact that $|\T| =(n-1)/r$, \eqref{eq:2r-2} implies that,
\begin{equation}
  \label{eq:size_of_B^*}
  |B^*| + |B^*_0| =2|B^*_0|+|B^*_1|=2|\mathcal T^*| \ge 2(|\T| - 3\tau n) \ge \left(1 - \frac{\gamma'}{4}\right)\frac{2n}{r}.
\end{equation}
So $|B^*| < (1 - \gamma')\frac{2n}{r}$ implies that $|B^*_0| > \gamma' n/r$, and
$|B^*_0| < (1 - \gamma')\frac{n}{r}$ implies that $|B^*| > (1 + \gamma'/2)\frac{n}{r}$.
Therefore, each of the at least $\gamma' n/r$ vertices 
in $B^*_0$ is incident to at least $|B^*| + \delta(G)/2 - n \ge (\gamma' n)/(4r)$ useful edges. In total we have at least $(1/2) \times(\gamma' n)/(4r) \times (\gamma' n)/r \gg 6 r\tau  n^2$ useful edges in $G$ which ensures we find our desired pair $\{T,T'\}$.


Now that we have a useful edge between a good pair, the next two claims show that this is sufficient to give us the desired perfect $(\K_r'\cup \K_{r+1}')$-tiling.

\begin{claim}\label{2r-2switch}
Let $T, T'\in \T^*$ be distinct and let $X\subseteq T$ and $X'\subseteq T'$ such that for all $x\in X$, $d(x, T')=2r-2$ and for all $x'\in X'$, $d(x', T)=2r-2$.  
\begin{enumerate}
\item If there exist $x\in X$ and $x'\in X'$ such that $\mu(xx')=0$, then $T-x+x'$ and $T'-x'+x$ are both copies of $\bar{\K}_r$.

\item If there exist $x\in X$ and $x'\in X'$ such that $\mu(xx')=1$, then  $T-x+x'$ and $T'-x'+x$ are both copies of $\hat{\K}_r$.  

\item If the bipartite graph of light edges induced by $X,X'$ is $2$-regular, then  $T-X+X'$ and $T'-X'+X$ are both copies of $\bar{\K}_r$.
\end{enumerate}

\end{claim}
The claim follows
 immediately  if $\mu(xx')=0$.  If $\mu(xx')=1$, then each vertex $x,x'$ has one other light neighbour, each of which would create a $\hat{\K}_r$ after the switch.  In the last case, all of the light neighbours of each vertex $x,x'$ are being moved to the other side.


\begin{claim}\label{caseanalysis}
Suppose $T, T'\in \T^*$ are distinct and there is a useful edge $bb'$ where $b \in T$, $b' \in T'$ such that $d(b,T')=2r-2$ and $d(b',T)=2r-2$.
Further suppose that for all $w\in T$, if $wb'$ is not heavy, then $d(w, T')=2r-2$ and for all $w'\in T'$ if $w'b$ is not heavy, then $d(w', T)=2r-2$.
%
%
%
Then there is a $(\K_r'\cup \K_{r+1}')$-tiling in $G$ covering precisely the vertices in $ V(T) \cup V(T')\cup \{v^*\}$.
\end{claim}

To prove the claim, we split the argument into three cases.
  \\
  \noindent \textbf{Case 1:} There exists $w \in T\setminus \{b\}$ or $w' \in T'\setminus \{b'\}$ such that either $\mu(bw') = 0$ or $\mu(b'w) = 0$.
  Without loss of generality, suppose $\mu(b'w) = 0$.  Switch $b'$ and $w$.  By Claim \ref{2r-2switch}(i), $T-w+b'$ and $T'-b'+w$ are $\bar{\K}_r$s.  Since $bb'$ is a useful edge,  $v^*$ sends at least $2r-2$ edges to $T'-b'+w$.
	If $v^*$ sends at least $2r-1$ edges to $T'-b'+w$ then 
	 $T'-b'+w+v^*$ is a copy of $\hat{\K}_{r+1}$.   If $v^*$ sends  precisely $2r-2$ edges to $T'-b'+w$ then 
	$w$ must be a light neighbour of $v^*$ in $G$. Further, as $\mu(b'w) = 0$ and $d(w, T')=2r-2$, we have that $w$ sends all possible edges to $T'-b'+w$, i.e. $d(w, T'-b'+w)=2r-2$.  In particular, $(\alpha)$ and $(\beta)$ do not hold (where $T'-b'+w$ is playing the role of $T$).
	Thus, $T'-b'+w+v^*$ is a copy of ${\K}'_{r+1}$.

  \noindent \textbf{Case 2:} $b\in B_0^*$ or $b'\in B_0^*$.\\
Without loss of generality, suppose $b\in B_0^*$. Since we are not in the first case, $b$ has two light neighbours in $T'$. In particular, there exists $x' \in V(T') \setminus B_0^*$ that is a light neighbour of $b$.  Switch $b$ and $x'$.  By Claim \ref{2r-2switch}(ii), $T-b+x'$ and $T'-x'+b$ are copies of $\hat{\K}_r$.  In particular, $T-b+x'$ has the property that if $T-b+x'$ contains a light path on $3$ vertices, then $x'$ is an endpoint of this path.  Furthermore since $x'\notin B_0^*$ and $\mu(v^*b)=0$, $v^*$ sends at least $2r-1$ edges to $T-b+x'$. Moreover, if $d(v^*,T-b+x')=2r-1$, then $\mu(v^*x')=1$ and thus $T-b+x'+v^*$ is a copy of $\K_{r+1}'$.

\begin{figure}[ht]
\centering
\includegraphics[scale=.85]{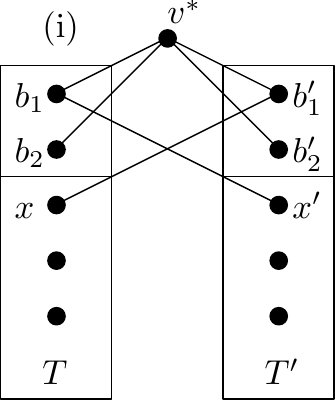}\hspace{.625in}
\includegraphics[scale=.85]{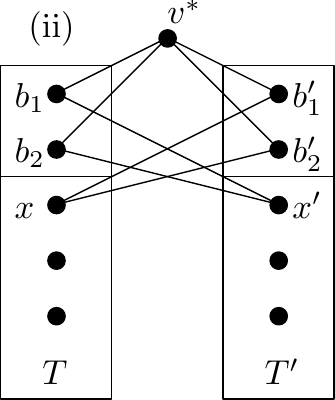}\hspace{.625in}
\includegraphics[scale=.85]{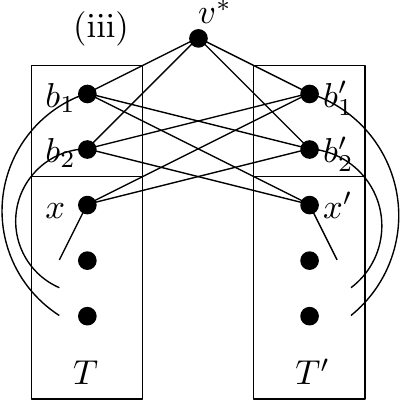}
\caption{Case 3: Note that the light edges in $T$ and $T'$ (not shown) form a matching.}
  \label{fig:case3}
\end{figure}

\noindent \textbf{Case 3:} $b,b'\in B_1^*$. \\
Let $B_1^*\cap T=\{b_1, b_2\}$ and $B_1^*\cap T'=\{b_1', b_2'\}$ with $b_1 = b$ and $b_1' = b'$. 
By the definition of a useful edge, $\mu(b_1b_1') = 2$.
Since we are not in Case 1, there exists $x'\in V(T')\setminus B_1^*$ that is a light neighbour of $b_1$ and
there exists $x \in V(T) \setminus B_1^*$ that is a light neighbour of $b_1'$ (see Figure~\ref{fig:case3}(i)).
Since $T' - x' + b_1 \in \hat{\K}_r$, we may assume that $T - b_1 + x' + v^* \notin \K'_{r+1}$ -- this implies that the other light neighbour of $x'$ in $T$ must be $b_2$ and that $b_2$ must have a light neighbour in $T - b_1$.  
Similarly, we can assume $T' - b_1' + x + v^* \notin \K'_{r+1}$, so 
$xb_2'$ is a light edge and $b_2'$ has a (unique) light neighbour in $T' - b_1'$.
Therefore, both $b_1b_2$ and $b_1'b_2'$ are  heavy edges (see Figure~\ref{fig:case3}(ii)).

Suppose $b_1b_2'$ is not a light edge, so $b_1$ has a light neighbour $x_2' \in V(T - x') \setminus B_1^*$.
As in the previous case, since $T' - x_2' + b_1 \in \hat{\K}_r$ we may assume that
$T - b_1 + x_2' + v^* \notin \K'_{r+1}$.  So it must be that $x_2'b_2$ is a light edge. 
Now $b_1x'b_2x_2'$ forms a $4$-cycle of light edges and by Claim \ref{2r-2switch}(iii), 
we can switch $b_1,b_2$ for $x',x_2'$ and then add $v^*$ to $T - b_1 - b_2 + x' + x_2'$ to obtain disjoint copies of $\bar{\K}_r$ and $\bar{\K}_{r+1}$.
Likewise, we would be done if $b_1'b_2$ is not a light edge. 
So suppose both $b_1b'_2$ and $b'_1b_2$ are light edges (see Figure~\ref{fig:case3}(iii)). Then
 $b_1 b'_2 x b_1'b_2x'$ forms a $6$-cycle of light edges; 
we simultaneously switch $b_1, b_2, x$ for $b_1',b_2',x'$ to obtain two disjoint copies of $\bar{\K}_r$.
Recall $b_1b_2$ is a heavy edge so
they both have distinct light neighbours in $T$.
Hence, at most one of $b_1$ and $b_2$ is a light neighbour of $x$.
Therefore we can add $v^*$ to $T' - b'_1 - b'_2 - x' + b_1 + b_2 + x$
to form an element of $\K_{r+1}'$.
%
%
%
This completes the proof of the claim and thus the lemma.
\end{proof}

	We now combine Corollary~\ref{almostthm_cor} and Lemma~\ref{lem_oneleftover} to obtain the following result.

\begin{proposition}\label{almostthm_reducedstability}
  Let $n,r \in \mathbb N$ where $r \geq 2$ and $\eta, \gamma >0$ such that $0<1/n \ll \eta \ll \gamma \ll 1/r$ and let $G$ be a multigraph on $n$ vertices.  If $G$ is not $(1/r,\gamma)$-extremal and
$$\delta (G) \geq 2(1-1/r-\eta )n,$$
then $G$ contains a perfect $(\K_r'\cup \K_{r+1}')$-tiling.
\end{proposition}

\begin{proof}
Choose $\tau, \gamma '$ so that $\eta \ll \tau\ll \gamma' \ll \gamma\ll 1/r$.  
By Corollary \ref{almostthm_cor} there exists a $\bar{\K}_r$-tiling $\mathcal T$ covering all but at most $4 r^2 \eta n$ vertices. 
Set $U :=V(G) \setminus V(\mathcal T)$.

To construct a perfect $(K_r'\cup K_{r+1}')$-tiling in $G$, we perform the following iterative procedure. 
For each vertex $v^*$ uncovered by $\T$ we apply  Lemma~\ref{lem_oneleftover} once.
In each iteration, we modify at most three elements of $\T$.
Each time we apply Lemma~\ref{lem_oneleftover}, the multigraph under consideration is the subgraph of $G$ induced by
$V(\T')\cup \{v^*\}$ where $\T' \subseteq \T$ is the set of tiles in $\T$ that have not been modified in any of the previous steps.

Suppose we have performed this procedure for every vertex in some $U' \subseteq U$. 
Let $G'$ be the subgraph under consideration and note that $|G'| \ge n - 3r|U'| - (|U| - 1)$, so 
$$\delta(G') \ge 2\left(1-1/r-\eta\right)n - (n - |G'|) \geq 2\left(1-1/r-\tau\right)|G'|.$$
  Furthermore, if $S \subseteq V(G')$ and $||S| - |G'|/r| \le \gamma' |G'|$ or
  $||S| - 2|G'|/r| \le \gamma' |G'|$, then
  $||S| - n/r| \le \gamma n$ or $||S| - 2n/r| \le \gamma n$, respectively.
  Therefore, $G'$ is not $(1/r, \gamma')$-extremal, 
  because $\gamma n^2 \ge \gamma' |G'|^2$.
  Hence, we may apply Lemma \ref{lem_oneleftover} a total of $|U|$ times to complete the proof.  
\end{proof}
We now apply the regularity lemma together with Proposition~\ref{almostthm_reducedstability} to prove Theorem~\ref{almostthm_stability}.

{\noindent \bf Proof of Theorem~\ref{almostthm_stability}.}
Define additional constants $\eps, d$ and $M' \in \mathbb N$ so that $1/n \ll 1/M' \ll \eps \ll d \ll \alpha$.  Apply Lemma~\ref{2colordegreeform} with parameters $\eps, d$ and $M'$ to $G$ 
to obtain clusters $V_1, \dots , V_k$, an exceptional set $V_0$ and a pure multigraph $G'$.
Set $m:=|V_1|=\dots =|V_k|$.
 Let $\Gamma$ be the reduced multigraph of $G$ with parameters $\eps, d$ and $M'$. Lemma~\ref{inherit} implies that 
\begin{align*}
\delta (\Gamma) \geq 2(1-1/r-2\eta )k.
\end{align*}

Suppose that there exists $S \subseteq V(\Gamma)$ such that $||S|-k/r|<\gamma k/4$ and $e(\Gamma [S])<\gamma k^2 /4$. Let $S =\{V_{i_1} , \dots , V_{i_t}\}$ and $S':=V_{i_1} \cup \dots \cup V_{i_t}$.
Then $||S'|-km/r|<\gamma km/4$ and so by Lemma~\ref{2colordegreeform}(ii), $||S'|-n/r|<\gamma n$. Moreover, by Lemma~\ref{2colordegreeform}(iv) and the definition of the reduced multigraph $\Gamma$,
$e(G'[S'])< (\gamma k^2/4)\cdot2m^2 \leq \gamma n^2/2$. Thus, by Lemma~\ref{2colordegreeform}(iii), $e(G[S'])< \gamma n^2$, a contradiction as $G$ is not $(1/r,\gamma)$-extremal.
A similar argument shows that there is no set $S \subseteq V(\Gamma)$ such that $||S|-2k/r|<\gamma k/4$ and $e_2(\Gamma [S])<\gamma k^2 /4$.
Thus, $\Gamma$ is not $(1/r,\gamma/4)$-extremal.

Therefore, by Proposition~\ref{almostthm_reducedstability}, $\Gamma$ contains a perfect $(\K_r'\cup \K_{r+1}')$-tiling $\mathcal T$. Note that every induced subgraph of a copy of $\mathcal K'_{r+1}$ of size $r$ is itself a copy of $\mathcal K'_{r}$.
Since $\eps \ll d \ll \alpha$, by repeatedly applying Lemma~\ref{red} for each of the tiles in $\mathcal T$ we obtain a $\K_r'$-tiling in $G$ covering all but at most $\alpha n$ vertices, as required.
\endproof

\section{The Absorbing Lemma}\label{secabs}
Let $G$ be a multigraph and $\mathcal H$ be a collection of multigraphs.
We call a set $M \subseteq V(G)$ an \emph{$\mathcal H$-absorbing set for $W \subseteq V(G)$} if both $G[M]$ and $G[M\cup W]$ contain perfect $\mathcal H$-tilings.
Suppose that a non-extremal multigraph $G$ as in Theorem~\ref{mainthm2} contains a small set $M \subseteq V(G)$ that is a $\mathcal K'_r$-absorbing set for any very small set $W \subseteq V(G)$.
Theorem~\ref{almostthm_stability} ensures that $G\setminus M$ contains an almost perfect  $\mathcal K'_r$-tiling; let $W$ denote the set of uncovered vertices in $G\setminus M$. Then $G[M\cup W]$ contains a perfect $\mathcal K'_r$-tiling, and thus $G$ contains a 
perfect $\mathcal K'_r$-tiling, as required.

We will show that such an absorbing set $M$ exists if $G$ is non-extremal and if additionally, in the case when $r=4$, $G$ is not `splittable'.  We use the following definition to make this precise.
\COMMENT{TM NEW: added definition environment for splittable.}
\begin{definition}
  \label{def:splittable}
  Given $\gamma >0$ and $r =2,4$, 
  a multigraph $G$ on $n$ vertices is \emph{\textit{$(1/r,\gamma)$-splittable}}
  if there exist disjoint sets $U_1,U_2 \subseteq V(G)$ such that \COMMENT{AT:  tweaked def of splittable}
  \begin{itemize}
    \item
      $|U_1|, |U_2| \ge (1/2 - \gamma)n$ and; 
    \item
      if $r=2$ then $e(U_1, U_2) \le \gamma n^2$; if $r=4$ then $e_2(U_1, U_2) \le \gamma n^2$.
  \end{itemize}
\end{definition}
The next result gives a  condition which forces a multigraph to contain an absorbing set.

\begin{lemma}[Lo and Markstr\"om~\cite{lo}]\label{lo}
Let $h,t \in \mathbb N$ and let $\gamma >0$. Suppose that $\mathcal H$ is a collection of multigraphs, each on $h$ vertices. Then there exists an $n_0 \in \mathbb N$ such that the following holds. Suppose that $G$ is a multigraph
on $n \geq n_0$ vertices so that, for any $x,y \in V(G)$, there are at least $\gamma n^{th-1}$ $(th-1)$-sets $X \subseteq V(G)$ such that both $G[X \cup \{x\}]$ and $G[X \cup \{y\}]$ contain perfect $\mathcal H$-tilings.
Then $V(G)$ contains a set $M$ so that 
\begin{itemize}
\item $|M|\leq (\gamma/2)^h n/4$;
\item $M$ is an $\mathcal H$-absorbing set for any $W \subseteq V(G) \setminus M$ such that $|W| \in h \mathbb N$ and  $|W|\leq (\gamma /2)^{2h} hn/32 $.
\end{itemize}
\end{lemma} 
Lo and Markstr\"om~\cite{lo} proved Lemma~\ref{lo} for hypergraphs, however, the proof of this result for multigraphs  is identical. 
The next result together with Lemma~\ref{lo} implies our multigraph $G$ contains an absorbing set.
\begin{lemma}
  \label{main-absorb}
  Let $r \ge 3$, $0 < 1/n \ll \eta , \phi \ll \gamma \ll 1/r$, and let
  $G$ be a multigraph on $n$ vertices.  
  If $\delta(G) \ge 2(1 - 1/r - \eta)n$ and $G$ is not $(1/r, \gamma)$-extremal and
  either $r \neq 4$ or $G$ is not $(1/r, \gamma)$-splittable, 
  then for all distinct $x_1,x_2 \in V(G)$ there exist
  at least $(\phi n)^{r-1}$ $(r-1)$-sets $Y\subseteq V(G)$ such that
  $G[Y \cup \{ x_1\}]$ and $G[Y \cup \{ x_2\}]$ both contain $\K_r'$.
\end{lemma}

We will need the following lemma in the proof of Lemma~\ref{main-absorb}.\COMMENT{AT: changed things in statement and proof of this lemma to reflect now $r=4$ only}
\begin{lemma}\label{prop:heavy_triangle_count}
  Let $0 < 1/n \ll \eta \ll \lambda \ll \gamma \ll 1$, and let
  $G$ be a multigraph on $n$ vertices with $\delta(G) \ge (3/2 - 2\eta)n$ which
  is not $(1/4, \gamma)$-extremal and not $(1/4, \gamma)$-splittable. 
  For any disjoint sets $U_1, U_2 \subseteq V(G)$ such that 
  $|U_1|, |U_2| \ge (1/2 - \gamma/5)n $, there exists 
  a collection $\T$ of copies of $\bar{\K}_3$ in $G[U_1 \cup U_2]$ such that
  $|\T| \ge \lambda n^3$ and 
  for every $T \in \mathcal{T}$, both $V(T) \cap U_1$ and $V(T) \cap U_2$ are non-empty, 
  and if $T$ contains a light edge $u_1u_2$, then $u_1\in U_1$ and $u_2\in U_2$.
\end{lemma}
\begin{proof}
  For $T\in \bar{\K}_3$ such that $T\subseteq G[U_1 \cup U_2]$, 
  we say that $T$ is \emph{nice} if both $V(T) \cap U_1$ and $V(T) \cap U_2$ are non-empty, and either $T$ has no light edges or 
  $T$ has exactly one light edge $u_1u_2$ with $u_1\in U_1$ and $u_2\in U_2$.
  Since $G$ is not $(1/4, \gamma)$-splittable, there exist at least $\gamma n^2$ heavy edges in $G$
  with one endpoint in $U_1$ and one endpoint in $U_2$.
  For each such edge $u_1u_2$, we will either find (i) at least $\gamma n/2$ vertices $u$ such that $u_1u_2u$ is a 
  nice $\bar{\K}_3$, or (ii) at least $\gamma n^2$ edges $e$ such that $u_ie$ is a nice $\bar{\K}_3$
  for some $i \in [2]$.  
	A simple calculation then implies that we obtain our desired collection of nice $\bar{\K}_3$.

  Let $u_1u_2$ be an edge in $E_2(U_1, U_2)$ such that $u_i \in U_i$ for $i \in [2]$.
  Pick $i$ so that $d^2(u_i) \ge d^2(u_{3-i})$ and note that
  \begin{equation*}
    (3/2 - 2\eta)n \le |N^2(u_{3-i})| + |N(u_{3-i})| \le |N^2(u_i)| + |N(u_{3-i})|,
  \end{equation*}
  so $|N^2(u_i) \cap N(u_{3-i})| \ge n/2  - 2 \eta n$. 
  Since for every $u \in N^2(u_i) \cap N(u_{3-i}) \cap U_i$, $u_1u_2u$ is a nice $\bar{\K}_3$,
  if $|N^2(u_i) \cap N(u_{3-i}) \cap U_i| \ge \gamma n/2$ we are done. 
  Otherwise, 
  \begin{align*}
    |N^2(u_i) \cap U_{3-i}| &\ge 
    |N^2(u_i) \cap N(u_{3-i}) \cap (U_1 \cup U_2)| -  |N^2(u_i) \cap N(u_{3-i}) \cap U_i|  \\
    &\ge \left[n/2 - 2 \eta n + 2(1/2 - \gamma/5)n - n \right] - \gamma n/2 \ge (1/2 - \gamma)n.
  \end{align*}
  Therefore, since $G$ is not $(1/4, \gamma)$-extremal,
  there are at least $\gamma n^2$ heavy edges in $G[N^2(u_i) \cap U_{3-i}]$ and
  for each such edge $e$, $u_ie$ is a nice $\bar{\K}_3$.
\end{proof}
We are now ready to prove Lemma~\ref{main-absorb}.

{\noindent \bf Proof of Lemma~\ref{main-absorb}.}
Define $\lambda$ so that $\eta,\phi\ll \lambda \ll \gamma$.
  Fix distinct vertices $x_1$, $x_2\in V(G)$ and let $X := \{x_1, x_2\}$.
  For any $U \subseteq V(G)$ (with $0 \le |U|\leq r$) and integer $i \ge 0$, let 
  \begin{equation*}
    S_i(U) := \{v \in V(G)\setminus U : d(v, U) \ge 2|U| - i\}.
  \end{equation*}
  Note that when $U = \emptyset$, we trivially have $S_i(U) = V(G)$.
  By the degree condition,
  \begin{equation*}
    (2r- 2)|U|n/r - 2|U| \eta n-|U|(|U|-1) \le e_G(U, V(G)\setminus U) \le |S_0(U)| + |S_1(U)| + (2|U| - 2)n.
  \end{equation*}
	Therefore,
  \begin{equation}
    \label{eq:U_0_plus_U_1}
    |S_0(U)| + |S_1(U)| \ge 
    (2r - 2|U|)n/r  - 3|U|\eta n.  
  \end{equation}
  So since $|S_1(U)| \le n$ and $S_0(U) \subseteq S_1(U)$,
  \begin{align}
    \label{eq:U_0}
    |S_0(U)| &\ge (r - 2|U|)n/r - 3|U| \eta n\text{ and } \\
    \label{eq:U_1}
    |S_1(U)| &\ge (r - |U|)n/r - 3|U|\eta n/2.
  \end{align}

  Call an $(r-1)$-set $Y\subseteq V(G)$ \emph{good} if both $G[Y + x_1]$ and $G[Y + x_2]$ contain $\K'_r$. 
  For $t\geq 0$ and $0\leq l\leq \floor{t/2}$, we say that a $t$-set $Y \subseteq S_0(X)$ is \emph{$l$-acceptable} if $G[Y]$ has exactly $l$ light edges and either:
  \begin{itemize}
    \item $t = 0$, 
    \item $t > 0$ and $G[Y] \in \bar{\K}_t$, or 
    \item $t = r - 3$ and $G[Y] \in \hat{\K}_t$. 
  \end{itemize}
  If $Y$ is $l$-acceptable for some $0\leq l\leq \floor{t/2}$, then we say that $Y$ is \emph{acceptable}; note that if $Y$ is an acceptable $t$-set, then both $G[Y + x_1]$ and $G[Y + x_2]$ contain $\hat{\K}_{t+1}$. 
For any acceptable $t$-set $Y$, let $S^*_1(Y)$ be the set of vertices $v \in S_1(Y) \cap S_0(X)$ such that if $y$ is the unique light neighbour of $v$ in $Y$, then $y$ is incident to a light edge in $Y$.
Note that if $Y$ is an $l$-acceptable $t$-set, $v \in S^*_1(Y)$ and $y$ is unique light neighbour of $v$ in $Y$, 
then $Y' := Y - y + v$, is either $(l-1)$-acceptable or $(l-2)$-acceptable depending on whether $y$ is incident to one or two light edges in $Y$.
Recall that $y$ can only be incident to two light edges in $Y$
if $t = r- 3$ and $y$ is the middle vertex of a path on three vertices
in $Y$ that consists of light edges.

For $0\leq t\leq r-3$ and $0\leq l\leq \floor{t/2}$, say that an $l$-acceptable $t$-set $Y$ is \emph{$\lambda$-extendible} if 
at least one of the following four conditions holds:
\begin{enumerate}
  \item $t = 0$, $r$ is even and $|S_0(X)| \ge \lambda n$;
  \item $t = 0$, $r=4$, and there are at least $(\lambda n)^3$ $3$-sets $Z$ such that $Z$ is a good $3$-set;
  \item $(r-1) - t$ is even, and 
    there are at least $(\lambda n)^2$ $2$-sets $Z$ such that $Y \cup Z$ is either an acceptable $(t+2)$-set 
    or a good $(r-1)$-set; or
  \item $|S_1^*(Y)| \ge \lambda n$.
\end{enumerate}
If we assume that (iv) never holds, then Claim~\ref{claim:extendible} below will imply that there exists at least $(\phi n)^{r-1}$ good $(r-1)$-sets.
Indeed, in this case 
\begin{itemize}
  \item
    if $r$ is odd, then we can build good $(r-1)$-sets two vertices at a time by repeatedly using (iii);
  \item
    if $r$ is even and (i) holds, then we can construct good $(r-1)$-sets by first selecting 
    any of the $\lambda n$ vertices in $S_0(X)$ and then finish the construction by repeatedly applying (iii);  and
  \item
    if $r$ is even and (i) does not hold, then (ii) must hold which immediately implies that there are $(\lambda n)^{r-1}$ good $(r-1)$-sets.
\end{itemize}

\begin{claim}\label{claim:extendible}
  Let $0\leq t\leq r-3$ and $0\leq l\leq \floor{t/2}$, and $Y$ be an $l$-acceptable $t$-set.
  If $t = 0$ or $(r-1) - t$ is even, then $Y$ is $\lambda$-extendible.
\end{claim}
To prove the claim, we may
  assume that (iv) does not hold throughout, i.e. 
  \begin{equation}
    \label{eq:not_v}
    |S_1^*(Y)| < \lambda n.
  \end{equation}
  Note that, by \eqref{eq:U_0},
  \begin{equation}
    \label{eq:size_of_S_0_X}
    |S_0(X)| \ge (r-4)n/r - 6\eta n.
  \end{equation}

  First assume that $(r-1) - t$ is even and let
  \begin{align*} 
    U := S_0(X) \cap S_0(Y) = S_0(X \cup Y).
  \end{align*} 
  Also define 
  \begin{align*} 
    U' &:= \left(S_0(X) \cap S_1(Y)\right) \setminus S_1^*(Y) && \text{when $t \le r - 5$, or} \\
    U' &:= S_1(X \cup Y) \setminus S_1^*(Y) && \text{when $t = r - 3$.}
  \end{align*}
  When $t \le r - 5$,
  \eqref{eq:U_0_plus_U_1}, \eqref{eq:not_v} and \eqref{eq:size_of_S_0_X} imply that
  \begin{equation}
    \label{eq:U_plus_Uprime}
    \begin{split}
      |U| + |U'| &\ge 
      (|S_0(Y)| + |S_0(X)| - n) + (|S_1(Y)| - |S^*_1(Y)| + |S_0(X)| - n) \\
      &\ge 2(r - t - 4)n/r - 2 \lambda n.
    \end{split}
  \end{equation}

  Assume that $t \le r - 7$.
  If $|U| \ge \gamma n/2$, then first pick any $z \in U$, 
  and then pick any vertex $z' \in N^2(z) \cap U'$
  and note that $Y + z + z'$ is an acceptable $(t+2)$-set.
  By \eqref{eq:U_plus_Uprime}, the minimum degree condition and the fact that
  $U \subseteq U'$, we have that there are at least
  \begin{equation*}
    3n/r - \lambda n - (2/r + 2\eta) n \ge n/r - 2\lambda n
  \end{equation*}
  choices for $z'$. Note that there are at least $(\gamma n/2)\times (n/r-2\lambda n)\times (1/2) >(\lambda n)^2$ choices for $\{z,z'\}$, so (iii) in the definition of $\lambda$-extendible holds, as required.
  If $|U| < \gamma n/2$, then
  pick any $z \in U' \setminus U$ and let $y \in Y$ be the unique light neighbour of $z$
  in $Y$.  Recall that since $z \notin S^*_1(Y)$, 
  $y$ has no light neighbours in $Y$.
  By \eqref{eq:U_plus_Uprime}, $|U' \setminus U| \ge 6n/r - 3 \lambda n$ and 
  for any of the at least $2n/r - 2 \gamma n$ vertices 
  $z' \in N^2(z) \cap N^2(y) \cap U'$,
  $Y + z + z'$ is an acceptable $(t+2)$-set. Note that there are at least \COMMENT{AT: this was a $3$ but think it should be a 6}$(6n/r - 3\lambda n)\times (2n/r - 2 \gamma n)\times (1/2) >(\lambda n)^2$ choices for $\{z,z'\}$, so (iii) in the definition of $\lambda$-extendible holds, as required.

  When $t = r - 5$, \eqref{eq:U_plus_Uprime} implies that
  \begin{equation*}
    |U| + |U'| \ge 2n/r - 2\lambda n,
  \end{equation*}
  and when $t = r - 3$, 
  \eqref{eq:U_0_plus_U_1} and \eqref{eq:not_v} give that
  \begin{equation*}
    |U| + |U'| \ge |S_0(X\cup Y)| + |S_1(X\cup Y)| -  |S_1^*(Y)| \ge
    2n/r - 3(r-1)\eta n - \lambda n.
  \end{equation*}
  Therefore, when $t \in \{r-3, r-5\}$, 
  \begin{equation}
    \label{eq:U_plus_Uprime_t_large}
    |U| + |U'| \ge 2n/r - \gamma n / 3.
  \end{equation}
  We will either find at least $\lambda n^2$
  light edges $zz'$ where $z \in U$ and $z' \in U'$, or at least
  $\lambda n^2$ heavy edges in $G[U']$.
  Note that, in either of these two cases,
  when $|Y| = r-5$,
  $G[Y + z + z' + x_i] \in \hat{\K}_{r-2}$ for $i \in [2]$,
  and when $|Y| = r-3$,
  $G[Y + z + z' + x_i] \in {\K}'_{r}$ for $i \in [2]$,
  so this will prove the claim.
  Suppose that we cannot find at least $\lambda n^2$ such edges.
  By non-extremality, this implies
  that $|U| < (1/r - \gamma)n$
  and $|U'| < (2/r - \gamma)n$.
  Hence, by \eqref{eq:U_plus_Uprime_t_large}, 
  $|U| \ge \frac{2\gamma n}{3}$ and $|U'| \ge (1/r + \frac{2}{3}\gamma)n$.
  Therefore, by the degree condition, any vertex in $z \in U$, is adjacent to at least 
  ${\gamma n}/{2}$ vertices $z' \in U'$, a contradiction.

  We will now show that if $r$ is even and $t = 0$ and (i) in the definition of $\lambda$-extendible does not hold, then (ii) must hold.
  So assume $t = 0$, $r$ is even and $|S_0(X)| < \lambda n$.  
  This, with \eqref{eq:U_0}, implies that $r = 4$ so $G$ is not $(1/r, \gamma)$-splittable.
  By the degree condition, $|S_0(X)| < \lambda n$ implies that 
  \begin{equation*}
    2n/r - 2\eta n \le |N^2(x_i)| \le 2n/r + 2\lambda n \text{ and } |N(x_i) |\ge (1 - 3\lambda)n \text{ for $i \in [2]$}.
  \end{equation*}
  Therefore, if we let
  \begin{equation*}
    U_i := (N^2(x_i) \cap N(x_{3-i})) \setminus N^2(x_{3-i}) \text{ for $i \in [2]$},
  \end{equation*}
  then $U_1$ and $U_2$ are disjoint, and
  \begin{equation*}
    |U_i| \ge (|N^2(x_i)| + |N(x_{3-i})| - n) - |S_0(X)| \ge (2/r - 5\lambda)n \text{ for $i \in [2]$}.
  \end{equation*}
  By Lemma~\ref{prop:heavy_triangle_count}, there are at least $\lambda n^3$ 
	nice copies of $\bar{\K}_3$ in $G[U_1\cup U_2]$ (i.e., copies of $\bar{\K}_3$
  that intersect both $U_1$ and $U_2$ and that have at most one light edge and such a light edge has one endpoint in $U_1$ and the other in $U_2$).  
  The vertex set of any such nice $\bar{\K}_3$ is a good $(r-1)$-set, as desired. This completes the proof of the claim.

\medskip
Assume for a contradiction that there are at most $(\phi n)^{r-1}$ good $(r-1)$-sets.
Let $0 \le t \le r-3$ be the maximum $t$ for which we have $(\phi n)^t$ acceptable $t$-sets
and either $t = 0$ or $(r-1) - t$ is even.
Such a $t$ exists, since $\emptyset$ is an acceptable $0$-set. 
We can also assume that $t > 0$, since if $Y = \emptyset$, then $S^*_1(Y) = \emptyset$ and (iv) cannot hold,
so Claim~\ref{claim:extendible} implies that one of (i), (ii) or (iii) must hold, which violates the maximality of $t$ or
the assumption that there are at most $(\phi n)^{r-1}$ good $(r-1)$-sets.
Let $l$ be minimal such that if $\mathcal{Y}$ is the set of $l$-acceptable $t$-sets, then
$|\mathcal{Y}| \ge \lambda(\lambda/16)^{\floor{t/2}-l}(\phi n)^t$.
There exists such an $l$, because $\lambda \ll 1/r$. 
By Claim~\ref{claim:extendible}, each set $Y\in \mathcal{Y}$ is $\lambda$-extendible in one of two ways,
so in particular there is some subset of $\mathcal{Y}$ of order at least $|\mathcal{Y}|/2$ for which all elements are extendible in the same way, 
either (iii), or (iv).
If the elements in this subset are all $\lambda$-extendible by (iii), then,
because $\phi \ll \lambda \ll 1/r$, we have at least
$$\frac{(\lambda n)^{2} \cdot |\mathcal{Y}|/2}{\binom{t+2}{2}} \ge 
\frac{\lambda^{2}/2 \cdot \lambda(\lambda/16)^{\floor{t/2}}}{\binom{t+2}{2}} \cdot \phi^{-2} \cdot (\phi n)^{t+2} >
(\phi n)^{t+2}$$ acceptable $(t+2)$-sets or, if $t=r-3$, more than  $(\phi n)^{r-1}$ good $(r-1)$-sets.\COMMENT{AT: changed slightly}
This contradicts the maximality of $t$  or the assumption that
there are at most $(\phi n)^{r-1}$ good $(r-1)$-sets.
(We divide by $\binom{t+2}{2}$ in the above calculation to account for the fact there are 
$\binom{t+2}{2}$ different ways a $(t+2)$-set can be constructed by adding $2$ vertices to a $t$-set.)

Now assume that there are at least $|\mathcal{Y}|/2$ $l$-acceptable $t$-sets that are 
extendible by (iv).
Let $\mathcal{Y}'$ be the collection of $(l-1)$-acceptable and $(l-2)$-acceptable $t$-sets and
set $$\mathcal{Z} := \{ Y + z : Y \in \mathcal{Y} \text{ and } z \in S_1^*(Y) \}.$$
Our aim is to find a lower bound on $|\mathcal{Y}'|$ which
contradicts the minimality of $l$.

Let $Z \in \mathcal{Z}$ and let 
$Y \in \mathcal{Y}$ and 
$z \in  S_1^*(Y)$
such that $Y + z = Z$.
Note that the choice of $Y$ is not necessarily unique.
Let $z'$ be the unique light neighbour of $z$ in $Y$ and
note that $Y' = Y - z' + z$ is in  $\mathcal{Y}'$.
Therefore, for every $Z \in \mathcal{Z}$ there exists
$Y' \in \mathcal{Y}'$ and $z' \in V(G)$, 
such that $Z = Y' + z'$ which implies that 
$|\mathcal{Y'}| \cdot n \ge |\mathcal{Z}|$.

Any $Z \in \mathcal{Z}$ is constructed by adding to some $Y \in \hat{\K}_t$ a vertex $z \in S_1^*(Y)$.
So in any $Z \in \mathcal{Z}$ there are most four vertices $z' \in Z$ such that there exists $z'' \in Z$
such that $z'z''$ is a light edge and $z''$ has at least two light neighbours in $Z$.
Therefore, 
for every $Z \in\mathcal{Z}$,
there are at most four different pairs $(Y, z)$ such that $Y \in \mathcal{Y}$, $z \in S_1^*(Y)$ and
$Z = Y + z$.
This implies that
$$|\mathcal{Z}| \ge \frac{|\mathcal{Y}|/2 \cdot \lambda n}{4} \ge (\lambda/8) \cdot \lambda (\lambda/16)^{\floor{t/2}-l}(\phi n)^t \cdot n,$$
so $|\mathcal{Y}'| \ge (\lambda/8) \cdot \lambda (\lambda/16)^{\floor{t/2}-l}(\phi n)^t$.
Therefore, there are at least 
$$\lambda(\lambda/16)^{\floor{t/2}-(l-1)}(\phi n)^t$$ 
$(l-1)$-acceptable $t$-sets or
$(l-2)$-acceptable $t$-sets, a contradiction to the minimality of $l$.
\endproof

\section{The stability result}
We now combine Lemmas~\ref{lo} and~\ref{main-absorb} together with Theorem~\ref{almostthm_stability} to prove the following result which ensures that Theorem~\ref{mainthm2} holds in the case when $G$ is non-extremal and additionally if $r=4$,  non-splittable.

\begin{theorem}\label{mainthm-stability}
  Let $n,r \in \mathbb N$ where $r $ divides $n$ and define $\eta, \gamma>0$ such that $0<1/n \ll \eta \ll \gamma \ll 1/r$.
  Let $G$ be a multigraph on $n$ vertices.  
  If $\delta(G) \ge 2(1 - 1/r - \eta)n$ and $G$ is not $(1/r, \gamma)$-extremal and
  either $r \notin \{2, 4\}$ or $G$ is not $(1/r, \gamma)$-splittable, 
  then $G$ contains a perfect $\K_r'$-tiling.
\end{theorem}
\begin{proof}
  First assume that $r \ge 3$.
  Define $\alpha, \phi, \eta', \gamma' >0$ so that $0<1/n \ll \alpha \ll \phi \ll \eta \ll \eta' \ll \gamma' \ll \gamma$.
  Let $G$ be as in the statement of the theorem. By Lemma~\ref{main-absorb}, given any $x_1,x_2 \in V(G)$, 
  there exist at least $(\phi n)^{r-1}$ $(r-1)$-sets $Y \subseteq V(G)$ such
  that both $G[Y \cup \{x_1\}]$ and $G[Y\cup \{x_2\}]$ contain elements of $\K'_r$. 
  Thus, by Lemma~\ref{lo}, $V(G)$ contains a set $M$ so that
  \begin{itemize}
    \item $|M| \leq ((\phi)^{r-1}/2)^r n/4$;
    \item $M$ is a $\K'_r$-absorbing set for any 
      $W \subseteq V(G) \setminus M$ such that $|W| \in r \mathbb N$ and  
      $|W| \leq ((\phi)^{r-1}/2)^{2r}rn/32$.
  \end{itemize}
  Let $G':=G\setminus V(M)$ and $n':=|G'|$. So as $\phi \ll \eta \ll \eta'$,
  $$\delta(G') \geq 2\left(1 - 1/r - \eta' \right)n'.$$
Further, as $\phi \ll \gamma' \ll \gamma$, $G'$ is not $(1/r, \gamma')$-extremal.
  Thus, by Theorem~\ref{almostthm_stability}, $G'$ contains a $\K'_r$-tiling covering all but at most 
  $\alpha n' < ((\phi)^{r-1}/2)^{2r} rn/ 32$ vertices. Let $W$ denote the set of these uncovered vertices. Then
  by definition of $M$, $G[M\cup W]$ contains a perfect $\K'_r$-tiling. 
  Altogether this implies that $G$ contains a perfect $\K'_r$-tiling, as desired.

  Now assume that $r \le 2$. 
   If $r = 1$ the theorem is trivial.
   For $r = 2$, we show 
   that there exists a perfect matching in $H$, the graph underlying the multigraph $G$,
  i.e.\ $H$ is the graph on $V(G)$ in which there is an edge between $x$ and $y$ if and only if there is either
  a light or heavy edge between $x$ and $y$ in $G$.
  We have that $n$ is even, $\eta \ll \gamma \ll 1/2$, $\delta(H) \ge (1/2 - \eta)n$, 
  $G$ is not $(1/2, \gamma)$-extremal and $G$ is not $(1/2, \gamma)$-splittable.  
Let $M$ be a maximum matching in $H$ and suppose $M$ is not perfect.  Note that the vertices unsaturated by $M$ form an independent set.  Let $w_1, w_2$ be two vertices unsaturated by $M$ and for $i\in [2]$, define $S_i:=\{v: uv\in M \text{ and } u\in N(w_i)\}$.  Note that $|S_1|, |S_2|\geq (1/2 - \eta)n$ and there are no edges in $H$ with one endpoint in $S_1$ and the other in $S_2$ as this would give us an $M$-augmenting path, contradicting the maximality of $M$.  Since $G$ is not $(1/2, \gamma)$-splittable, it cannot be the case that $S_1$ and $S_2$ are disjoint, so let $v\in S_1\cap S_2$.  We have $N(v)\cap (S_1\cup S_2)=\emptyset$ and thus $|S_1\cup S_2|\leq (1/2+\eta)n$ which implies $|S_1\cap S_2|\geq (1/2-3\eta)n\geq (1/2-\gamma)n$, contradicting the fact that $G$ is not $(1/2, \gamma)$-extremal.  
  
\end{proof}


\section{The extremal case}\label{extremal}

In this section we prove  the following theorem.
\begin{theorem}
  \label{thm:main_extremal}
  For any $r \in \N$, there exists $n_0\in \mathbb N$ such that
  the following holds.
  If $G$ is a multigraph on $n \ge n_0$ vertices,
  $n$ is divisible by $r$ and 
  \begin{equation}
    \label{eq:precise}
    \delta(G) \ge 2(1 - 1/r)n - 1,
  \end{equation}
  then $G$ contains a perfect $\mathcal{K}'_r$-tiling.
\end{theorem}
Note that Theorem~\ref{thm:main_extremal} immediately implies Theorem~\ref{mainthm2} (and thus Theorem~\ref{mainthm}).
Our results from the previous sections will ensure Theorem~\ref{thm:main_extremal} holds in the `non-extremal' cases. Therefore most of the work in this section
concerns the extremal cases.

Throughout this section we consider a standard multigraph $G$ on $n$ vertices that satisfies the hypothesis of Theorem~\ref{thm:main_extremal}.
In particular, we may assume $1/n\ll 1/r$. We denote the vertex set of $G$ by $V$.

%

\subsection{Preliminary claims}

We will use the following well-known and simple lemma in this section.
A proof is included for completeness.
\begin{lemma}
  \label{lem:matching}
  For any graph $H$
  there is a matching of order at least 
  $\min\{\floor{|H|/2}, \delta(H)\}$.
\end{lemma}
\begin{proof}
  Let $M$ be a maximal matching in $H$ and suppose that
  $|M| < \min\left\{\floor{|H|/2}, \delta(H)\right\}$.
  Let $U$ be the vertices incident to an edge in $M$ and let $W := \overline{U}$.
  Note that $|W| \ge 2$ and that, by the maximality of $|M|$, 
  $W$ is an independent set.
  Therefore, there exist distinct $x,y \in W$ where
  $d(x, U) + d(y, U) \ge 2 \delta(H) > 2|M|$. 
  Hence, there exists $e \in M$ such that $d(x, e) + d(y, e) \ge 3$, 
  and this implies that there exists a matching of order $|M| + 1$ in $H$.
\end{proof}
The next claim gives us a minimum degree condition for $G[U]$ where  $U$ is any set of size close to $sn/r$ for some $s\in \mathbb N$.
\begin{claim}
  \label{clm:deg_into_subsets}
  Suppose $0<1/n \ll c \ll 1/r$ and $s \in \mathbb N$ where 
  $1 \le s \le r$. Let $v \in V$ and $U \subseteq V$.
  If 
  \begin{equation*}
    sn/r - cn \le |U| \le sn/r + cn, 
  \end{equation*}
  then 
  \begin{equation*}
    d(v, U) \ge 2(1 - 1/s - rc/s)|U|.
  \end{equation*}
\end{claim}
\begin{proof}
  When $s = 1$ the statement is trivially true, so assume $s \ge 2$.
  Therefore, 
  \begin{equation*}
    (1/s + rc/s)|U| \ge (1/s + rc/s)(s/r - c)n 
    = (1/r + (s-1)c/s - rc^2/s)n > n/r + 1.
  \end{equation*}
  Hence, by \eqref{eq:precise},
  \begin{equation*}
    d(v, U) 
    \ge \delta(G) - 2|\overline{U}|
    \ge 2(|U| - (n/r + 1))
    > 2(1 - 1/s - rc/s)|U|. \qedhere
  \end{equation*}
\end{proof}


  Let $c$ be a constant such that $0 < c < 1$.
  We call a set $U \subseteq V$,
  \begin{itemize}
    \item a \emph{$(1, c)$-independent set} if $|U| = n/r$ and $e(G[U]) < c n^2$, or
    \item a \emph{$(2, c)$-independent set} if $|U| = 2n/r$ and $e_2(G[U]) < c n^2$.
  \end{itemize}
  If $U,U' \subseteq V$ are vertex-disjoint, we say that
  the pair $\{U,U'\}$ is 
  \begin{itemize}
    \item a \emph{$(1, c)$-disconnected pair} if $|U| = |U'| = n/r$ and $e(U, U') < c n^2$, or
    \item a \emph{$(2, c)$-disconnected pair} if $|U| = |U'| = 2n/r$ and $e_2(U, U') < c n^2$.
  \end{itemize}
  If $1 \le s \le r$ and $U \subseteq V$, we call $U$ an \emph{$(s, c)$-tolerant set} when $|U| = sn/r$ and
  \begin{itemize}
    \item for $t \in \{1, 2\}$, $U$ does not contain a $(t, c)$-independent set, and
    \item if $s \in \{2,4\}$, then $U$ does not contain an $(s/2, c)$-disconnected pair.
  \end{itemize}

The preceding definitions are closely related to the notion
of being $(1/r, \gamma)$-extremal (Definition~\ref{def:extremal}) or $(1/r, \gamma)$-splittable 
(Definition~\ref{def:splittable}) which
the following simple claim makes explicit.
\begin{claim}
  \label{clm:tolerant-implies-not-extremal}
  Let $1 \le s \le r$ where $s \in \mathbb N$, 
  $0<1/n \ll c' \ll \gamma \ll c \ll 1/r$,
  $U \subseteq V$ such that $|U| = sn/r$ and
  $U' \subseteq U$ such that 
  $|U \triangle U'| \le c' n$.
  If $U \subseteq V$ is $(s,c)$-tolerant,
  then $G[U']$ is not $(1/s,\gamma)$-extremal 
  and, when $s \in \{2, 4\}$,
  $G[U']$ is not $(1/s, \gamma)$-splittable.
\end{claim}
\begin{proof}
  Suppose that $G[U']$ is $(1/s, \gamma)$-extremal.
  So there exists 
  $W \subseteq U'$ such that 
  either
  $e(G[W]) < \gamma |U'|^2$, 
  and $|W| \ge (1/s - \gamma) |U'| \ge (1/r - 2\gamma)n$, or
  $e_2(G[W]) < \gamma |U'|^2$, 
  and $|W| \ge (2/s - \gamma) |U'| \ge (2/r - 2\gamma)n$.
  Because $c' \ll \gamma \ll c \ll 1/r$ and $|U \triangle U'| \le c' n$, it
  is easy to see we can add vertices from $U$ to $W \cap U$ or delete
  vertices from $W \cap U$ to create 
  either a $(1, c)$-independent or $(2, c)$-independent set in $U$.
  This implies that $U$ is not $(s, c)$-tolerant.

  A similar argument implies that when $s \in \{2, 4\}$, 
  if $G[U']$ is $(1/s, \gamma)$-splittable, then $U$ is not
  $(s, c)$-tolerant.
\end{proof}

Claim~\ref{clm:tolerant_sets} below is meant to capture all of the necessary
facts about $(s, c)$-tolerant sets in a form that will be convenient.
In some sense, Claim~\ref{clm:tolerant_sets} is just a restatement of 
the main theorems from the previous sections.

With Claim~\ref{clm:deg_into_subsets}, 
(i) and
(ii)
follow from Corollary~\ref{almostthm_cor},
and (iii) and (iv)
follow from Proposition~\ref{almostthm_reducedstability} and 
Theorem~\ref{mainthm-stability}, respectively.
Note that, after the proof of this claim,
we do not appeal to Corollary~\ref{almostthm_cor},
Proposition~\ref{almostthm_reducedstability} and 
Theorem~\ref{mainthm-stability} again.

\begin{claim}
  \label{clm:tolerant_sets}
  Let $1 \le s \le r$ where $s \in \mathbb N$ and
  suppose $1/n \ll \eta \ll \phi \ll \gamma \ll 1/r$,
  $U \subseteq V$ such that 
  $U$ is $(s, \gamma)$-tolerant and $U' \subseteq V$
  such that $|U \triangle U'| \le \eta n$.
  \begin{enumerate}
    \item
      \label{clm:Kbar_in_Uprime}
      If $W \subseteq U'$ such that $|W| \ge (s-1)n/r + \phi n$
      then $G[W]$ contains a copy of $\bar{\K}_{s}$.
    \item
      \label{clm:v_Khat_s_plus_1}
      If $v \in V$ such that $d(v, U') \ge 2(s-1)n/r + \phi n$,
      then $G[U']$ contains a copy $T$ of $\bar{K}_s$ to which 
      $v$ sends at least $2s-1$ edges, so
      $G[T + v]$ contains a copy of $\hat{\K}_{s+1}$.
    \item
      \label{clm:large_cliques}
      There exist at least $\phi n$ vertex-disjoint copies of $\K'_{s+1}$ in $G[U']$.
    \item
      \label{clm:factor}
      If $|U'|$ is divisible by $s$, 
      then there exists a perfect $\K'_s$-tiling in $G[U']$.
  \end{enumerate}
\end{claim}
\begin{proof}
  Note that, by Claim~\ref{clm:deg_into_subsets},
  \begin{equation}
    \label{eq:min_deg_H}
    \delta(G[U']) \ge 2(1 - 1/s - r\eta/s)|U'|.
  \end{equation}

  We first prove (i) and (ii),
  so let $W$ and $v$ be as in the statement of the claim.
  When $s \ge 2$, 
  Corollary~\ref{almostthm_cor} and \eqref{eq:min_deg_H}
  imply that there exists a
  $\bar{K}_s$-tiling $\mathcal{T}$ of $G[U']$ on all 
  but at most $4 s r \eta |U'|$ 
  vertices, and when $s = 1$, this is trivially true.
  Let $Z := V(\T)$
  be the vertices in $U'$ that are covered by $\mathcal{T}$.
  Note that 
  \begin{equation*}
    (s-1)|\T| \le (s-1)|U'|/s < (s-1)n/r + \phi n/3.
  \end{equation*}
  Since 
  \begin{equation*}
    |W \cap Z| \ge |W| - |U' \setminus Z| \ge 
    |W| - 4sr \eta |U'| > (s-1)n/r + \phi n/3 > (s-1)|\mathcal{T}|,
  \end{equation*}
  there exists $T \in \mathcal{T}$ such that $V(T) \subseteq W$,
  and this proves (i).
  Because 
  \begin{equation*}
    d(v, U') \ge 2((s - 1)n/r + \phi n/3) + \phi n/3 
    > 2(s - 1)|\mathcal{T}| + 2|U' \setminus Z| 
  \end{equation*}
  there exists $T \in \mathcal{T}$ such that $d(v, V(T)) \ge 2s - 1$,
  so $v$ has at most one light neighbour in $T$.
  This proves (ii).

  We will now prove (iii) and (iv). 
  To this end, let $U'' \subseteq U'$ such that 
  \begin{equation*}
    |U''| \ge |U'| - (s+1) \ceiling{ \phi n }.
  \end{equation*}
  Note that
  \begin{equation*}
    |U'' \triangle U| \le 
    |U'' \triangle U'| + |U' \triangle U| \le 4(s + 1)\phi n/3 + \eta n
    \le 3 s \phi n
  \end{equation*}
  so, by Claim~\ref{clm:deg_into_subsets},
  \begin{equation*}
    \delta(G[U'']) \ge 2(1 - 1/s - 3 r \phi)|U''|.
  \end{equation*}
  For $\gamma'$ such that 
  $1/n \ll 3r \phi \ll \gamma' \ll \gamma \ll 1/s$, by Claim~\ref{clm:tolerant-implies-not-extremal}
  we have that 
  $G[U'']$ is not $(1/s, \gamma')$-extremal
  and when $s \in \{2, 4\}$
  is not $(1/s, \gamma')$-splittable.
	If $s=1$, (iv) is vacuously true. So suppose $s \geq 2$.
Then
  Proposition~\ref{almostthm_reducedstability} implies
  that $G[U'']$ has a perfect $(\K'_s \cup \K'_{s+1})$-tiling
  and, when $s$ divides $|U''|$,
  Theorem~\ref{mainthm-stability} implies
  that $G[U'']$ has a perfect $\K'_s$-tiling.

  Taking $U'' = U'$ then gives (iv).
  Furthermore, we can greedily select $\ceiling{ \phi n }$ 
  copies of elements from $\K'_{s+1}$ from $G[U']$,
  since any subset of $U'$ that has order greater than
  $|U'| - (s+1)\ceiling{ \phi n }$ contains
  a copy of $\K'_{s+1}$:
  When $s \ge 2$,
  this is true because a perfect $(\K'_s\cup \K'_{s+1})$-tiling in a multigraph of
  order not divisible by $s$ implies the existence of an element
  from $\K'_{s+1}$,
  and, when $s = 1$, this is true because
  $G[U']$ not being $(1, \gamma')$-extremal implies that
  $G[U']$ contains at least $\gamma' |U'|^2$ edges.
  This proves (iii). 
%
%
\end{proof}

\subsection {Initial partitioning and sorting}

Suppose that
\begin{equation*}
  0 < 1/n \ll \gamma_0 \ll \gamma_1 \ll \dotsm \ll \gamma_{r+2} \ll 1/r
\end{equation*}
and, 
in addition, for every $i \in [r+2]$, we have 
$\beta_{i}$ and $\psi_{i}$
such that
\begin{equation*}
  \gamma_{i-1} \ll \beta_{i} \ll \psi_{i} \ll \gamma_{i}.
\end{equation*}

We start by trying to find, for either $s_1=1$ or $s_1 = 2$,
an $(s_1, \gamma_1)$-independent set 
which we will call $A_1$.
We then try to find, for $s_2 = 1$ or $s_2 = 2$, an 
$(s_2, \gamma_2)$-independent set $A_2$ disjoint from $A_1$.
We continue in this manner for as long as possible,
so in the end we have (a possibly empty, in which case $p=0$) collection of 
disjoint sets $A_1, \dotsc, A_p$ and integers
$s_1, \dotsc, s_p$, such
that $A_i$ is $(s_i, \gamma_i)$-independent for each $i \in [p]$.
Let $U := \overline{\bigcup_{i \in [p]} A_i}$ and
set $s := r - \left(s_1 + \dotsm + s_p\right)$.
If $s = 0$, then $U = \emptyset$ 
and we set $q:=0$ to indicate this case.
If $U$ is $(s, \gamma_{p+1})$-tolerant,
then set $A_{p+1} := U$, $s_{p+1} := s$ and $q := 1$.
Otherwise, $U$ is not $(s, \gamma_{p+1})$-tolerant
and because the initial process terminated, $U$ contains neither a $(1, \gamma_{p+1})$-independent set nor a $(2, \gamma_{p+1})$-independent set.
By the definition of a tolerant set, 
it must therefore be that $s \in \{2, 4\}$
and that $U$ has a partition $\{A_{p+1}, A_{p+2}\}$
that is $(s/2, \gamma_{p+1})$-disconnected.
We set $s_{p+1}: = s_{p+2} := s/2$ and $q := 2$, to indicate this case.

If $q = 1$, set
$\tau := \gamma_{p}$,
$\beta := \beta_{p+1}$,
$\psi := \psi_{p+1}$ and
$\gamma := \gamma_{p+1}$.
Otherwise, set
$\tau := \gamma_{p+1}$,
$\beta := \beta_{p+2}$,
$\psi := \psi_{p+2}$ and
$\gamma := \gamma_{p+2}$.
Therefore, 
\begin{equation}
  \label{eq:constants}
  0 < 1/n \ll \tau \ll \beta \ll \psi \ll \gamma \ll 1,
\end{equation}
and we have proved the following claim.
\begin{claim}
  \label{clm:first_partition}
  There exists a partition $\{A_1, \dotsc, A_{p+q}\}$ of $V$
  where $p + q \le r$ and $q \in \{0, 1, 2\}$,
  and non-negative integers $s, s_1, \dotsc, s_{p+q}$ 
  such that the following holds:
  \begin{enumerate}
      \item
        $|A_i| = s_i n/r$ for $i \in [p+q]$;
      \item
        for every $i \in [p]$, $s_i \in \{1, 2\}$ and 
        $A_i$ is an $(s_i, \tau)$-independent set;
      \item
        if $q = 0$, then $s = 0$;
      \item
        \label{clm:properties_of_A_p+1}
        if $q = 1$, then 
        $s = s_{p+1}$ and
        $A_{p+1}$ is $(s, \gamma)$-tolerant; and
      \item
        if $q = 2$, then $s \in \{2, 4\}$ and
        $\{A_{p+1}, A_{p+2}\}$ is an $(s/2, \tau)$-disconnected pair.
    \end{enumerate}
  \end{claim}
Note the relationship between Claim~\ref{clm:first_partition} and the examples shown in Figure~\ref{fig:extremalcases}.  \COMMENT{LD: I wanted to note this somewhere and here seemed like a good place.}
   
    In order to discuss the case when $q=2$ and the case when $q\neq2$ in a consistent
    way, we define a permutation $\sigma$ of $[p+q]$ in the following way.
    If $q \neq 2$ we let $\sigma$ be the identity permutation,
    and if $q = 2$ we let $\sigma$ be the transposition of $p+1$ and $p+2$.
    Note that when $p+1 \le i \le p+q$, and $q \ge 1$, 
    $|A_i \cup A_{\sigma(i)}| = sn/r$.

    When $q \neq 1$, we let $\Lambda := [p+q]$,
    and when $q = 1$, we let $\Lambda := [p]$.
    We say that a vertex $v$ is \emph{$(i,c)$-typical} if 
    $i \in \Lambda$ and 
    \begin{equation*}
      \text{$s_i = 1$ and $d(v, A_{\sigma(i)}) \le cn$, or
      $s_i = 2$ and $d^2(v, A_{\sigma(i)}) \le cn$;}
    \end{equation*}
    or, if $q = 1$, $i = p+1$ and
    $d^2(v, \overline{A_{p+1}}) \ge |\overline{A_{p+1}}| - c n$.

  \begin{claim}
    \label{clm:typical}
    For any $i \in [p+q]$, if $v$ is $(i,c)$-typical, then
    $d^2(v, \overline{A_{\sigma(i)}}) \ge |\overline{A_{\sigma(i)}}| - cn - 1$,
    and, furthermore, if $i \in \Lambda$ and $s_i = 2$, then $|N(v)| \ge (1-c)n - 1$.
  \end{claim}
  \begin{proof}
Consider any $i \in [p+q]$ and suppose $v$ is $(i,c)$-typical. 
    If $q=1$ and $i = p+1$, then, recalling that $i = \sigma(i) = p+1$, we have 
    $d^2(v, \overline{A_{\sigma(i)}}) \ge |\overline{A_{\sigma(i)}}| - c n$ by definition.

    Otherwise, using \eqref{eq:precise}, if $s_i = 1$, we have that
    \begin{equation*}
      d^2(v, \overline{A_{\sigma(i)}}) \ge 
      \delta(G) - |\overline{A_{\sigma(i)}}| - d(v, A_{\sigma(i)})
      \ge |\overline{A_{\sigma(i)}}| - cn  - 1
    \end{equation*}
    and if $s_i = 2$, we have that 
    \begin{equation*}
      d^2(v, \overline{A_{\sigma(i)}}) \ge 
      \delta(G) - (n-1) - d^2(v, A_{\sigma(i)}) \ge
      |\overline{A_{\sigma(i)}}| - cn.
    \end{equation*}
    When $s_i = 2$, we also have that
    \begin{equation*}
      |N(v)| \ge \delta(G) - d^2(v) \ge \delta(G) - |\overline{A_{\sigma(i)}}| - cn \ge (1 - c)n - 1. 
      \qedhere
    \end{equation*}
  \end{proof}

  \begin{claim}
    \label{clm:nontypical_bound}
    For any $0 <c \le 1$ and $i \in [p+q]$, 
    there are at most $3r \tau n/c$ vertices in $A_i$ that
    are not $(i, c)$-typical.
  \end{claim}
  \begin{proof}
    Let $i \in [p+q]$ and let $t$
    be the number of vertices in $A_i$ that are not $(i,c)$-typical.
    If $i \in \Lambda$, then we have that
    \begin{alignat*}{3}
      t c n &\le e(A_i, A_{\sigma(i)}) \le 2\tau n^2   &&\text{ when $s_i = 1$, and,} \\
      t c n &\le e_2(A_i, A_{\sigma(i)}) \le 2\tau n^2 &&\text{ when $s_i = 2$,}
    \end{alignat*}
    so $t \le 2 \tau n/c$. Here we define $e(A_i,A_i):=2e(A_i)$ and $e_2(A_i,A_i):=2e_2(A_i)$.
		\COMMENT{AT: added sentence.  LD: Should we just define it this way in the notation section? AT NEW: I deliberately put it here as it is only used here and I though the reader would forget this subtlety if it was defined 25 pages back}

    If $q = 1$ and $i = p+1$, then our assumption gives us
    \begin{equation*}
      |A_{p+1}||\overline{A_{p+1}}| - t c n \ge  e_2(A_{p+1}, \overline{A_{p+1}}),
    \end{equation*}
   so if we can show that 
    \begin{equation}
      \label{eq:A_p_plus_edge_one_lower_bound}
      e_2(A_{p+1}, \overline{A_{p+1}}) \ge |A_{p+1}||\overline{A_{p+1}}| - 
      3p\tau n^2,
    \end{equation}
    this will imply $t \le 3p \tau n/c$, which will prove the claim.
    To show \eqref{eq:A_p_plus_edge_one_lower_bound},
    let $j \in [p]$, and recall that $j = \sigma(j)$
    and $|A_j| \in \{n/r, 2n/r\}$.
    Clearly, 
    \begin{equation*}
      e_2(A_j, \overline{A_j}) \ge 
      |A_j|\delta(G) - e(A_j, A_j) - |A_j||\overline{A_j}|.
    \end{equation*}
    Therefore, when $|A_{j}| = n/r$, 
    $\delta(G) \ge 2 |\overline{A_j}| - 1$,
    and $e(A_j, A_j) \le 2\tau n^2$, so
    \begin{equation*}
      e_2(A_j, \overline{A_j}) \ge 
      |A_j||\overline{A_j}| - |A_j| - 2 \tau n^2 \ge 
      |A_j||\overline{A_j}| - 3 \tau n^2.
    \end{equation*}
    When $|A_j| = 2n/r$, $\delta(G) \ge 2|\overline{A_j}| + |A_j|  - 1$, 
    and $e(A_j, A_j) \le |A_j|^2 + 2\tau n^2$, so
    \begin{equation*}
      e_2(A_j, \overline{A_j}) \ge  |A_j||\overline{A_j}| + |A_j|(|A_j| - 1) - (|A_j|^2 + 2 \tau n^2)
      \ge  |A_j||\overline{A_j}| - 3 \tau n^2.
    \end{equation*}
    Therefore,
    \begin{equation*}
      e_2(A_{p+1}, \overline{A_{p+1}}) = \sum_{j = 1}^{p} e_2(A_{p+1}, A_j) \ge \sum_{j=1}^{p} \left(|A_{p+1}||A_j| - 3 \tau n^2\right)
      = |A_{p+1}||\overline{A_{p+1}}| - 3p \tau n^2.
    \end{equation*}
  \end{proof}

    Let $\mathcal{U} = (U_1, \dotsc, U_{p+q})$ be an ordered collection of $p+q$
    pairwise disjoint subsets of $V$.
    We say that an $r$-set $T \subseteq \bigcup_{i=1}^{p+q} U_i$ 
    is \emph{$\mathcal{U}$-balanced}, if 
    \begin{itemize}
      \item $G[T]$ contains a copy of $\K'_r$, 
      \item $|V(T) \cap U_i| = s_i$ for all  $i \in [p]$, and 
      \item $|V(T) \cap \left(U_{p+1} \cup U_{\sigma(p+1)}\right)| = s$ when $q > 0$.
    \end{itemize}
    A set $T$ is called \emph{$\mathcal{U}$-well-balanced}, 
    if $T$ is $\mathcal{U}$-balanced,
    and when $q = 2$, $T$ intersects exactly one of the
    two sets $U_{p+1}$ and $U_{p+2}$, i.e.\
    for some $i \in \{p+1, p+2\}$, 
    $|V(T) \cap U_i| = s$ and 
    $|V(T) \cap U_{\sigma(i)}| = 0$.
    Note that when $q \neq 2$ every $\mathcal{U}$-balanced set
    is  a
    $\mathcal{U}$-well-balanced set.


  We say that a vertex is \textit{excellent for $i$}
  if it is $(i, \beta^2)$-typical
  and we say that a vertex is \textit{good for $i$}
  if it is $(i, \psi^2)$-typical.
  We make the following definitions:
  \begin{equation*}
    B_i  := \{v \in A_i : v \text{ is excellent for $i$} \} \text{ and }
    B := \bigcup_{ i \in [p+q] } B_i,
  \end{equation*}
  and we let 
  \begin{equation*}
    C_i  := B_i \cup \{v \in \overline{B} : v \text{ is good for $i$} \} 
    \text{ and }
    C := \bigcup_{ i \in [p+q] } C_i.
  \end{equation*}
  We let $\mathcal{B}$ be the ordered collection $(B_1, \dotsc, B_{p+q})$
  and $\mathcal{C}$ be the ordered collection $(C_1, \dotsc, C_{p+q})$.
  Note that, by Claim~\ref{clm:typical}, every vertex 
  is good for at most one index $i \in [p+q]$, 
  so the sets $C_1, \dotsc, C_{p+q}$ are pairwise disjoint.
  Note that, for every $i \in [p+q]$, 
  each vertex $v \in \overline{C}$ is not good
  for $i$, so it has a large number of edges into $B_{\sigma(i)}$.\COMMENT{AT: changed $i$ to $\sigma(i)$}
  Since each vertex $v\in B_i$ is adjacent to almost everything
  in $\overline{B_\sigma(i)}$, we can argue below that there exists 
  a set $T$ that is $\mathcal{B}$-well-balanced and such
  that $G[T + v]$ contains a copy of $\K'_{r+1}$.
  As we will see, this will allow us to easily distribute the vertices in $\overline{C}$
  to sets in $\mathcal{C}$.

  Claims~\ref{clm:typical} and~\ref{clm:nontypical_bound}
  and the fact that $\tau \ll \beta \ll \psi \ll 1/r$
  imply the following claim.
  \begin{claim}
    \label{clm:second_partition}
    The following holds:
    \begin{enumerate}
      \item
	\label{clm:size_B_i_C_i}
        $|\overline{C}| \le |\overline{B}| \le \beta^2 n$, and,
	in particular, for every $i \in [p+q]$,
        \begin{equation*} 
          \text{$|A_i \triangle C_i|,
            |A_i \triangle B_i| \le \beta^2 n$
          and $|C_i \setminus B_i| \le \beta^2 n$;}
        \end{equation*}
      \item
	\label{clm:excellent}
	for every $i \in [p+q]$, \\
        $v \in B_i \implies 
          d^2(v, \overline{B_{\sigma(i)}}) \ge |\overline{B_{\sigma(i)}}| - \beta n$, and  \\
          $v \in C_i \implies
          d^2(v, \overline{B_{\sigma(i)}}) \ge |\overline{B_{\sigma(i)}}| - \psi n$;
      \item
	\label{clm:s_i_equals_2}
	for every $i \in \Lambda$, if $s_i = 2$, then 
        \begin{equation*}
          \text{ $v \in B_i \implies |N(v)| \ge (1 - \beta) n$ and 
            $v \in C_i \implies |N(v)| \ge (1 - \psi) n$.
          }
        \end{equation*}
    \end{enumerate}
  \end{claim}

  Looking ahead, we will construct an ordered collection 
  $\mathcal{D} = (D_1, \dotsc, D_{p+q})$ such that
  $\{D_1, \dotsc, D_{p+q}\}$ is a partition of $V$ and
  such that there exists a perfect $\mathcal K_{r}'$-tiling in $G$ such that 
  every element in the tiling is $\mathcal{D}$-well-balanced.
  This trivially implies that $\mathcal{D}$ must have the following properties:
  \begin{enumerate}
    \item 
      \label{proper_std}
      for every $i \in [p]$ and for $i = p+1$ when $q = 1$,
      $|D_i| = s_i \cdot n/r$,
    \item 
      \label{proper_q_equals_2}
      when $q = 2$,
      both $|D_{p+1}|$ and $|D_{p+2}|$ are divisible by $s$
      and $|D_{p+1} \cup D_{p+2}| = s n/r$.
  \end{enumerate}
  For any multigraph $G' \subseteq G$, 
  call $\mathcal{D} = (D'_1, \dotsc, D'_{p+q})$
  a \emph{proper ordered collection of $G'$} 
  if the sets in $\mathcal{D}$ form a partition
  of $V(G')$ and it meets conditions 
  (\ref{proper_std}) and (\ref{proper_q_equals_2}) 
  (with $n$ replaced by $|G'|$).
  Let $c_i := |C_i| - s_in/r$ for every $i \in [p+q]$,
  so $|\overline{C}| + c_1 + \dotsm + c_{p+q} = 0$. 

  \bigskip

  In order to make the rest of the section easier to understand,
  we now give a brief, informal description of 
  the remainder proof for the case when $q \neq 2$.
  It is very similar to the approach taken by
  Koml{\'o}s, S{\'a}rk{\"o}zy and Szemer{\'e}di in their
  proof of the Alon--Yuster conjecture \cite{komlos2001proof}.
  We ignore the case when $q=2$ in this description
  to avoid technical details.

  Our main goal is to balance the sizes of the sets $C_1, \dotsc, C_{p+q}$.
  We begin by considering the sets that are too large, i.e.\
  for every $i \in [p+q]$ such that $c_i > 0$, 
  we move exactly $c_i$ vertices out of $C_i$
  to form the set $D_i$.
  We want to ensure that the vertices which are moved
  out of $C_i$ can eventually be covered by a copy of $\K'_r$
  that has $s_i$ other vertices in $C_i$.
  Therefore, we argue that we can find 
  a $\K'_{s_i + 1}$-tiling in $G[C_i]$ of size $c_i$.
  From each element of this tiling we arbitrarily select a vertex $v$ 
  to remove from $C_i$ when forming $D_i$, and, temporarily, place $v$
  into a set we call $F'$.
  Then we extend each element of this tiling to form 
  a copy of $\K'_{r+1}$ that has
  exactly $s_j$ vertices in $C_j$ for each $j \in [p+q] \setminus i$,
  and use the label $T_v$ for this copy of $\K'_{r+1}$.
  So $T_v - v$ is $\mathcal{C}$-well-balanced \COMMENT{LD: Changed $T$ to $T_v$}.
  We let $\T = \{T_v : v \in F'\}$.
  Note that, after this process has completed, we have that,
  for every $c_i \ge 0$, 
  $D_i = C_i \setminus F'$ and $|D_i| = s_i n/r$.

  Next, we prepare to distribute the vertices that were not assigned
  to some set $C_i$. 
  To do this,
  for every such $v \in \overline{C}$, we find a $\B$-well-balanced
  $r$-set $T'$ such that $G[T' + v]$ contains an element of $\K'_{r+1}$.
  We then label $T' + v$ as $T_v$, and add $T_v$ to $\T$.
  Throughout, we ensure that the elements in $\T$ are vertex-disjoint.
  We let $F = F' \cup \overline{C}$.

  Note that by Claim~\ref{clm:second_partition}(\ref{clm:size_B_i_C_i}), 
  \begin{equation*}
    |F| = \sum_{i \in [p+q];~c_i < 0} -c_i \le 
    |\overline{B}| \le \beta^2 n.
  \end{equation*}
  So we can distribute the small number of vertices in $F$ 
  arbitrarily to every $D_i$ such that $|D_i| < s_i n/r$ until
  $|D_i| = s_in/r$ for every $i \in [p+q]$.
  Suppose $v \in F$ has been assigned to $D_i$;
  so $T_v$ has exactly $s_i$ vertices in $C_i$ and,
  with $v$, has $s_i + 1$ vertices in $D_i$.
  We can then arbitrarily remove one element from $T_v \cap C_i$
  to create a $\mathcal{D}$-well-balanced set.
  After this has been done for every $v \in F$, 
  we have that $\T$ is a $\K'_r$-tiling
  in which every element corresponds to a $\mathcal{D}$-well-balanced $r$-set.
  We let $G' = G - V(\T)$ and
  $D'_i = D_i \setminus V(\T)$ for every $i \in [p+q]$.
  So $(D'_1, \dotsc, D'_{p+q})$ is a proper ordered collection 
  of $G'$ and $D'_i \subseteq C_i$ for every $i \in [p+q]$.
  Claim~\ref{clm:proper_partitions} below will then complete the proof.
  When $q = 2$, our approach is similar.
  The main difference is that we have to be somewhat careful
  to ensure that $|D_{p+1}|$ and $|D_{p+2}|$ are both divisible by $s$. 
  The details for all cases are in Section~\ref{sec:finishing}.
%

  \bigskip

  We now continue the proof of Theorem~\ref{thm:main_extremal} 
  by proving Claim~\ref{clm:proper_partitions}
  which will be used at the very end of the proof to construct the 
  vast majority of elements of our tiling of $G$.
  \begin{claim}
    \label{clm:proper_partitions}
    Let $G'$ be an induced subgraph of $G$ such that 
    $|G'| \ge |G| - \beta n$,
    and let $\mathcal{D}' = (D'_1, \dotsc, D'_{p+q})$ be a proper 
    ordered collection of $V(G')$.
    If $D'_i \subseteq C_i$ for every $i \in [p+q]$, then there
    exists a perfect $\K'_r$-tiling in $G'$.
  \end{claim}
  \begin{proof}
    For $i \in [p]$, we let $\mathcal{T}_i$ be a perfect $\bar{\K}_{s_i}$-tiling of $G[D'_i]$.
    When $s_i=1$ such a tiling trivially exists, and, when $s_i=2$, it
    exists by Claim~\ref{clm:second_partition}(\ref{clm:s_i_equals_2}).
    When $q \ge 1$, we let $\mathcal{T}_{p+1}$ be a perfect $\K'_s$-tiling
    of $G[D'_{p+1} \cup D'_{\sigma(p+1)}]$.
    By Claim~\ref{clm:second_partition}(\ref{clm:excellent}),
    this is easy to find when $q = 2$ by, say, 
    applying the Hajnal--Szemer\'edi theorem
    \COMMENT{LD: Ideally we wouldn't use the Hajnal-Szemer\'edi theorem at all as we would like to say our result implies the original (for sufficiently large $n$).  Clearly its use is overkill here, but I understand the point of using it for simplicity.  Does anybody have any suggestions for how we can most easily avoid it?  Perhaps we can have a lemma which says that if $|G|$ is divisible by $r$ and $\delta(G)\geq (1-\ep)n$ for $\ep\ll 1/r$, then $G$ has a perfect $K_r$ tiling.  In other words, at what minimum degree does the Hajnal-Szemeredi theorem become easy to prove?} 
    to the graph induced by
    the heavy edges of $D'_{p+1}$\COMMENT{LD: Changed $D$ to $D'$.} and then
    to the graph induced by the heavy edges of $D'_{p+2}$.
		(Actually the use of the Hajnal--Szemer\'edi theorem here is overkill; it is very easy to directly argue the desired perfect tiling in $G[D'_{p+1} \cup D'_{\sigma(p+1)}]$ exists.)\COMMENT{AT NEW: added sentence}
    When $q=1$, it is implied by 
    Claim~\ref{clm:tolerant_sets}(\ref{clm:factor}), 
    because $\beta \ll \gamma$ and
    \begin{equation*}
      |A_{p+1} \triangle D'_{p+1}| \le 
      |A_{p+1} \triangle C_{p+1}| + (|G| - |G'|) \le 
      (\beta^2 + \beta)n .
    \end{equation*}

    Let $t = p$ when $q=0$, or let $t = p+1$, when $q \ge 1$, and
    let $H$ be a $t$-partite graph with vertex classes 
    $\T_1, \dotsc, \T_t$ such that, 
    for every distinct $i,i' \in [t]$, $T \in \mathcal{T}_i$ and
    $T' \in \mathcal{T}_{i'}$,
    $T$ is adjacent to $T'$ 
    when every vertex in $T$ is heavily adjacent to every vertex in $T'$.
    Note that $H$ is balanced with each vertex class of size $m = |G'|/r$.
    So we are done if there is a perfect $K_t$-tiling in $H$.
    By Claim~\ref{clm:second_partition}(\ref{clm:excellent}),
    when $i,i'$ are distinct and $T \in \mathcal{T}_i$, 
    the number of neighbours of $T$ in $\mathcal{T}_{i'}$ is at least 
    \begin{equation}
      \label{eq:deg_in_H}
      m -  \left| \bigcup_{u \in T} \left( D'_{i'} \setminus N^2(u)\right) \right| \ge 
      m - r \psi n \ge (1 - 2r^2 \psi)m.
    \end{equation}

    For some $1 \le t' < t$,
    assume we have  a perfect $K_{t'}$-tiling $\mathcal{K}$ of 
    $H[\mathcal{T}_1 \cup \dotsm \cup \mathcal{T}_{t'}]$.
    We extend $\mathcal{K}$ to a perfect $K_{t' + 1}$-tiling of
    $H[\mathcal{T}_1 \cup \dotsm \cup \mathcal{T}_{t'+1}]$
    by finding a perfect matching in the bipartite graph with vertex classes $\mathcal{K}$ and 
    $\mathcal{T}_{t' + 1}$ in which $K \in \mathcal{K}$ is adjacent to 
    $T \in \mathcal{T}_{t'+1}$ when $K + T$ is a copy of $K_{t'+1}$ in $H$.
    By \eqref{eq:deg_in_H}, this bipartite graph has minimum degree at least 
    $(1 - 2r^2 t' \psi)m \ge m/2$, so Hall's Theorem implies that it has
    a perfect matching.
    Since this holds for every $t' < t$, there exists  a perfect
    $K_t$-tiling in $H$.
  \end{proof}

  In Claim~\ref{clm:extend_partial_clique}(\ref{clm:extend_partial_clique_std}) below
  we establish that, for some $k \in [p+q]$, 
  we can extend a copy of $\K'_{s_k + 1}$ (resp. $\K'_{s_k}$)
  that is contained in some $C_k$ where $k \le p$ or 
  extend a copy of $\K'_{s+1}$ (resp. $\K'_{s}$)
  contained in $C_{p+1} \cup C_{\sigma(p+1)}$
  to a copy of $\K'_{r+1}$ (resp. $\K'_{r}$) 
  with the correct number of vertices 
  in every $C_j$ for $j \in [p+q] - i$.
  This is used when we remove elements from the sets
  $C_k$ that are too large
  as described in the overview above.
  When $q = 2$, we use 
  Claim~\ref{clm:extend_partial_clique}(\ref{clm:extend_partial_clique_std})
  with $t = 0$ when 
  $|C_{p+1} \cup C_{p+2}| \ge sn/r$
  to move vertices between $C_{p+1}$ to $C_{p+2}$ to make the
  order of both sets divisible by $s$.
  The second part of the lemma, (\ref{clm:extend_partial_clique_special}),
  which is only used when $q = 2$, 
  is similar but we start with a copy of $\K'_{s_k + 1}$ in some
  $C_k$ with $k \le p$ and extend it only into a copy of $\K'_r$.
  Furthermore, for any $0 \le \ell \le s - 1$,
  this $\K'_r$ will have exactly $\ell$ vertices
  in $C_{p+1}$ and $s - \ell - 1$ vertices in $C_{p+2}$.
  This is used because sometimes when we must move a vertex 
  from some $C_k$ with order greater than $s_k n/r$,
  to $C_{p+1}$ because the order of $C_{p+1} \cup C_{p+2}$ is 
  slightly less than $sn/r$, 
  we may also have to move some vertices in $C_{p+2}$ to 
  $C_{p+1}$ to ensure that both sets are divisible by $s$. 
  \begin{claim}
    \label{clm:extend_partial_clique}
    Let $W \subseteq V$ such that $|W| \le \beta n$, let $k \in [p+q]$
    and  $t \in \{0,1\}$.
    When $k \le p$, let $S \subseteq C_k$ such that $|S| = s_k + t$
    and when $k \ge p + 1$, let $S \subseteq C_k \cup C_{\sigma(k)}$ 
    such that $|S| = s + t$.
    Suppose that $G[S]$ contains an element of $\K'_{|S|}$.
    \begin{enumerate}
      \item
	\label{clm:extend_partial_clique_std}
	There exists a set $T$ that avoids $W$ such that
	$T \cup S$ is $\C$-balanced when $t = 0$, or,
	when $t = 1$, $T \cup (S - v)$ is $\C$-balanced for any $v \in S$.
        Furthermore, when $q = 2$, the set $T$ does not
        intersect $C_{p+2}$.
      \item
	\label{clm:extend_partial_clique_special}
	When $q = 2$, $t = 1$, and $k \notin \{p+1, p+2\}$, 
	for any $j \in \{p+1, p+2\}$ and $1 \le \ell \le s - 1$, 
	there exists a set $T$ that avoids $W$ 
	such that $G[T \cup S]$ contains an element of $\K'_r$,
	$|T \cap C_j| = \ell$, $|T \cap C_{\sigma(j)}| = s-1 -\ell$ and 
	$|T \cap C_i| = s_i$ for each $i \in [p] - k$.
    \end{enumerate}
  \end{claim}
  \begin{proof}
    We will construct $T$ iteratively, and 
    throughout, we let 
    \begin{equation*}
      C'_i := \left(C_i \setminus W \right) \cap \left(\bigcap_{u \in T \cup S} N^2(u) \right).
    \end{equation*}
    Assuming $T \subseteq C$, 
    Claim~\ref{clm:second_partition}(\ref{clm:excellent})
    implies that for all $i$ such that 
    $C_{\sigma(i)} \cap \left(S \cup T \right) = \emptyset$,
    \begin{equation}
      \label{eq:size_of_Cprime}
      |C'_i| \ge |C_i| - |W| - |T \cup S|\psi n \ge |C_i| - 2r \psi n. 
    \end{equation}

    We start the construction by 
    adding vertices from
    $C'_{p+1} \cup C'_{\sigma(p+1)}$ to $T$.
    If $q = 0$, or $q \ge 1$ and $k \ge p+1$, we do not
    add any vertices from $C'_{p+1} \cup C'_{\sigma(p+1)}$ to $T$.
    Otherwise, 
    we know 
    $|C'_{p+1} \cup C'_{\sigma(p+1)}|$ 
    is large by \eqref{eq:size_of_Cprime} 
    since $(C_{p+1} \cup C_{\sigma(p+1)}) \cap S = \emptyset$.
    If $q = 1$, then, since $\beta \ll \psi \ll \gamma$,
    Claim~\ref{clm:tolerant_sets}(\ref{clm:Kbar_in_Uprime}) and
    Claim~\ref{clm:second_partition}(\ref{clm:size_B_i_C_i})
    imply that we 
    can let $T \subseteq C'_{p+1}$ such that $|T| = s$ and 
    $G[T]$ contains an element of $\bar{\K}_s$.
    To prove (\ref{clm:extend_partial_clique_std}) when $q = 2$,
    note that
    \eqref{eq:size_of_Cprime} and 
    Claim~\ref{clm:second_partition}(\ref{clm:excellent}),
    imply that we can let $T$ induce a clique of size $s$ 
    on heavy edges in $C'_{p+1}$.
    To prove (\ref{clm:extend_partial_clique_special}),
    we start by using
    \eqref{eq:size_of_Cprime} and  
    Claim~\ref{clm:second_partition}(\ref{clm:excellent})
    to find vertices $u_1,\dotsc,u_{\ell} \in C'_j$ 
    that form a clique on heavy edges and
    add these vertices to $T$. 
    If $s - 1 - \ell = 0$, we are done, so assume that this 
    is not the case, which implies that $s = 4$.
    Therefore, 
    Claim~\ref{clm:second_partition}(\ref{clm:excellent}) and 
    Claim~\ref{clm:second_partition}(\ref{clm:s_i_equals_2}),
    imply that we can find $s-1-\ell$ vertices in 
    $N(u_1) \cap \dotsm \cap  N(u_\ell) \cap C'_{\sigma(j)}$ that
    form a clique on heavy edges. We then add
    these vertices to $T$. Note that in this case when $s=4$ either $\ell=1$ and $s-1-\ell=2$ or $\ell =2$ and $s-1-\ell=1$. So in this case currently $G[T]$ contains all possible edges except that it may have a light path on $3$ vertices.

    Now, for every $i \in \{1, \dotsc, p\} - k$, in turn 
    we use \eqref{eq:size_of_Cprime}, to add
    either one vertex, when $s_i = 1$, or 
    two adjacent vertices, when $s_i = 2$, from $C'_i$ to $T$.
    When $s_i = 2$ we can easily find two adjacent vertices in $C'_i$ using 
    Claim~\ref{clm:second_partition}(\ref{clm:s_i_equals_2}). Further, note that by definition of $C'_i$, in this step any selected vertices in $C'_i$ send heavy edges to all previously selected elements of $S \cup T$.
		This ensures that there are all possible edges in $G[S \cup T]$ except for perhaps a collection of vertex-disjoint light edges, and at most one path on $3$ vertices (in the case when $s=4$). That is, 
		$G[S\cup T]$ is a copy of $\mathcal K' _{|S\cup T|}$. It is now easy to see that the claim holds.
  \end{proof}

  The main purpose of 
  Claim~\ref{clm:extend_overline_C}(\ref{clm:extend_overline_C_std})
  is to help distribute the vertices  $v \in \overline{C}$.
  We construct a $\B$-well balanced set $T$ such that
  $G[T + v]$ contains an element of $\K'_{r+1}$; 
  thus we can then move
  $v$ to any set $D_i$ and remove any element in
  $B_i \cap T$ from $T$ leaving a $\D$-well-balanced set.
  When $q=2$, we may need to move
  some $v \in \overline{C}$ to one of $D_{p+j}$ for $j \in \{1, 2\}$ 
  and, at the same time, move some vertices
  from $D_{p+3-j}$ to $D_{p+j}$ to ensure that
  $D_{p+1}$ and $D_{p+2}$ are both divisible by $s$.
  This is the reason for 
  Claim~\ref{clm:extend_overline_C}(\ref{clm:extend_overline_C_special}).
  Note that the following claim is essentially the reason we define
  both the sets $B_{1}, \dotsc, B_{p+q}$ and
  the sets $C_{1}, \dotsc, C_{p+q}$. 
  \begin{claim}
    \label{clm:extend_overline_C}
    Let $W \subseteq V$ such that $|W| \le \beta n$ and $v \in \overline{C}$.
    \begin{enumerate}
      \item 
	\label{clm:extend_overline_C_std}
	There exists a $\B$-well-balanced set $T$ that avoids $W$ such that 
	$G[T + v] \in \K'_{r+1}$.
%
      \item
	\label{clm:extend_overline_C_special}
	When $q = 2$, for $j \in \{p+1, p+2\}$ and $1 \le \ell \le s - 1$, 
	there exists a set $T$ that avoids $W$ 
	such that $G[T + v]$ contains an element of $\K'_r$,
	$|T \cap B_j| = \ell$, $|T \cap B_{\sigma(j)}| = s - 1 - \ell$ and 
	$|T \cap B_i| = s_i$ for each $i \in [p]$.
    \end{enumerate}
  \end{claim}
  \begin{proof}
    For any $i \in [p+q]$, 
    if $i \in \Lambda$, 
    since $v$ it not $(i, \psi^2)$-typical,
    Claim~\ref{clm:second_partition}(\ref{clm:size_B_i_C_i})
    implies that 
    \begin{equation*}
      \text{when $s_i = 1$, $d(v, B_i) \ge \psi^2 n/2$, and, 
      when $s_i = 2$, $d^2(v, B_i) \ge \psi^2 n/2$.}
    \end{equation*}
    Similarly, when $q = 1$, since $v$ is not $(p+1, \psi^2)$-typical,
    $d^2(v, \overline{A_{p+1}}) \le |\overline{A_{p+1}}| - \psi^2 n$ and 
    \begin{equation}\label{eq:bad_vertices}
   \begin{split}
        d(v, B_{p+1}) &\ge \delta(G) - d(v, \overline{B_{p+1}}) \\
        &\ge \delta(G) - \left(|\overline{A_{p+1}}| + d^2(v, \overline{A_{p+1}}) + 
        2|\overline{B_{p+1}} \setminus \overline{A_{p+1}}|\right)  \\ 
        &\ge 2(1 - 1/s)sn/r + \psi^2 n/2.   
      \end{split}
    \end{equation}
    For $i \in [p+q]$, let $B'_{i} := B_i \setminus W$. We have that, by
    Claim~\ref{clm:second_partition}(\ref{clm:size_B_i_C_i}),
    \begin{equation}
      \label{eq:A_i_symdiff_B_i_prime}
      |A_i \triangle B'_i| \le |A_i \triangle B_i| + |W| \le 2\beta n.
    \end{equation}
    Let 
    \begin{equation*}
      I := \{i \in \Lambda : s_i = 1 \text{ and } d^2(v, B_i) \le \psi^2 n/2 \}
    \end{equation*}
    and note that there are at least 
    $\sum_{i \in I} \left(|B_i| - \psi^2 n/2\right) \ge |I|(1/r - \psi^2)n$
    vertices in $\bigcup_{i \in I} B_i$ that are not heavy neighbours
    of $v$.
    By \eqref{eq:precise}, when $r > 2$,  
    $d^2(v) \ge \delta(G) - (n-1) \ge r \psi n/2$, so there
    exists $j \in [p + q] \setminus I$, and
    \begin{equation*}
      d^2(v, B_{j}) \ge 
      \delta(G) - 
      (n - 1) - \left(|\overline{B_{j}}| - |I|(1/r - \psi^2)n \right)
      \ge |B_j|  - (2/r - |I|/r + |I|\psi^2)n.
    \end{equation*}
    Hence, in all cases, $|I| \le 2$, and, for any $j \in [p+q] \setminus I$,
    \begin{equation}
      \label{eq:I_equals_2}
      d^2(v, B_{j}) \ge |B_{j}| - 2 \psi^2 n \qquad \text{ if $|I| = 2$.}
    \end{equation}

    We prove (\ref{clm:extend_overline_C_std})
    and (\ref{clm:extend_overline_C_special}) simultaneously.
    Let $t \in \{0,1\}$,
    so that $s - 1 + t$ is the number of vertices 
    of $T$ that will intersect $B_{p+1} \cup B_{\sigma(p+1)}$, i.e.\ 
    when we are proving (\ref{clm:extend_overline_C_std}) we have
    $t = 1$ and 
    when we are proving (\ref{clm:extend_overline_C_special}) 
    we have $t = 0$. 

    We now give a brief overview of our plan  
    for constructing the set $T$.
    Clearly we must construct $T$ so that 
    every pair of vertices in $T + v$ are adjacent.
    We will also have that 
    \begin{itemize}
      \item[($\alpha _1$)] 
        the only light edges in $G[T + v]$ that
        are not incident to $v$,
        are in the subgraphs $G[T \cap B'_i]$ where $s_i = 2$,
        or are in $G[T \cap \left(B'_{p+1} \cup B'_{\sigma(p+1)}\right)]$
        when $q \ge 1$;
      \item[($\alpha _2$)] 
        $v$ is heavily adjacent to every vertex
        $u$ in $T \cap B'_i$ when $i \in \Lambda \setminus I$;\COMMENT{AT: was $i \notin \Lambda \setminus I$, but I think this is now correct}
     \item[($\alpha _3$)] 
        if $q \ge 1$,
        $G[T \cap \left(B'_{p+1} \cup B'_{\sigma(p+1)}\right) + v]$
        will contain an element of $\K'_{s - 1 + t}$.
    \end{itemize}
    If $v$ has at most one light neighbour in $T$,
    this is enough to give us that 
    $G[T + v]$ contains a an element of $\K'_{r+t}$,
    which would prove the claim.
    However, we can only 
    ensure that $v$ has at most two light neighbours in $T$.
    To prove the claim, we will then also meet one of the
    following conditions when $v$ has two light neighbours in $T$:
    \begin{itemize}
      \item[($\beta _1$)]
	the two light neighbours of $v$ are in different sets
	$B'_i$ such that $i \in I$ and $i \le p$
  and if $q \ge 1$, $G[T \cap (B'_{p+1} \cup B'_{\sigma(p+1)})]$ 
	contains  an element of $\overline{\K}_{s-1+t}$ except
	possibly when 
	$q=2$, $s=4$, $t = 0$
	and $\ell \in \{1, 2\}$, and in this case
	it contains an element of $\hat{\K}_{s-1+t}$; or
      \item[($\beta _2$)]
	$q = 2$, $s=2$, $t=1$, $I = \{p+1, p+2\}$ and the two light neighbours
	of $v$ are in some $B'_i$ where $i \in I$
	and these two light neighbours are heavily adjacent; or
      \item[($\beta _3$)]
	$q = 2$, $s=2$, $t=0$, $|I| = 2$, $j \in I \cap \{p+1, p+2\}$, 
	$v$ has one light neighbour in $B'_j$ and
        the other light neighbour is in some $B'_i$ where $i \in I$
	and $i \le p$; or
      \item[($\beta _4$)]
	$q=1$, $|I| = 1$, 
	one light neighbour of $v$ is in the set $B'_i$ such that $i \in I$
	and the other light neighbour of $v$ is in $B'_{p+1}$
	and $G[T \cap B'_{p+1} + v]$ contains an element of $\hat{\K}_{s+t}$.
    \end{itemize}

    We build the set $T$ iteratively.
    We start by adding vertices from  
    $B'_{p+1} \cup B'_{\sigma(p+1)}$ to $T$, so
    when $q = 0$ we do not add anything to $T$. 
    Recall that $\beta \ll \psi \ll \gamma$.
    If $q = 1$ and $|I| \le 1$, then
    \eqref{eq:bad_vertices},
    \eqref{eq:A_i_symdiff_B_i_prime} and
    Claim~\ref{clm:tolerant_sets}(\ref{clm:v_Khat_s_plus_1}), 
    imply that we can choose $T \subseteq B_{p+1}'$, such that 
    $|T| = s$, $G[T + v]$ contains an element of $\hat{\K}_{s+1}$ 
    and in which $v$ has one light neighbour.
    If $q = 1$ and $|I| = 2$, then by 
    \eqref{eq:A_i_symdiff_B_i_prime}, \eqref{eq:I_equals_2},
    and Claim~\ref{clm:tolerant_sets}(\ref{clm:Kbar_in_Uprime}) 
    we can let $T \subseteq B_{p+1}'$ be such that
    $T \subseteq N^2(v)$ and $G[T]$ contains an element of $\bar{\K}_s$. 

    Now assume $q = 2$.
    To prove (\ref{clm:extend_overline_C_std}),
    when $I = \{p+1, p+2\}$, we pick $j \in \{p+1, p+2\}$ arbitrarily,
    otherwise we let $j \in \{p+1, p+2\} \setminus I$.
    Recall that $j \in I$ implies that $s = 2$.
    Let $Z := N(v) \cap B'_j$ when $j \in I$
    and let $Z := N^2(v) \cap B'_j$ when $j \notin I$.
    Note that, 
    $|Z| \ge \psi^2 n/2 - |W| \ge \psi^2 n / 3$, 
    in either case.
    We now use Claim~\ref{clm:second_partition}(\ref{clm:excellent}) to
    find an $s$-set $T \subseteq Z$ such that $T$ induces a clique on heavy edges in $G$.
		
    To prove (\ref{clm:extend_overline_C_special}),
    we assume $j \in \{p+1, p+2\}$ is given.
    Note that if $j \in I$, then we must have that $s = 2$ and $\ell = 1$.
    Again, we
    let $Z: = N(v) \cap B'_j$ when $j \in I$
    and let $Z := N^2(v) \cap B'_j$ when $j \notin I$,
    so $|Z| \ge \psi^2 n / 3$. 
    Now we use Claim~\ref{clm:second_partition}(\ref{clm:excellent})
    to find vertices $u_1,\dotsc,u_{\ell} \in Z$ 
    such that they induce a clique on heavy edges in $G$, and we add
    these vertices to $T$. 
    If $s - 1 - \ell = 0$, we are done, so assume that this 
    is not the case, which implies that $s = 4$.
    Therefore, 
    Claim~\ref{clm:second_partition}(\ref{clm:s_i_equals_2}) and
    Claim~\ref{clm:second_partition}(\ref{clm:excellent}),
    imply that we can find $s-1-\ell$ vertices in 
    $N(u_1) \cap \dotsm \cap  N(u_\ell) \cap N^2(v) \cap B'_{\sigma(j)}$ that
    induce a clique on heavy edges in $G$. We then add
    these vertices to $T$.
    If $\ell = \{1, 2\}$, then $G[T]$ contains an element of $\hat{K}_3$
    and when $\ell = 3$ it is a clique on heavy edges.

Note that in all cases the way we have constructed $T$ so far ensures that ($\alpha _3$) holds.

%

    Now, in turn for each $i$ with $1\leq i \leq p$, we will add $s_i$ vertices
    from $B'_i$ to $T$. At each step,
    when $i \in I$, we let  
    \begin{equation*}
      Z := B'_i \cap N(v) \cap \left(\bigcap_{u \in T} N^2(u)\right),
      \text{ and otherwise let }
      Z := B'_i \cap N^2(v) \cap \left(\bigcap_{u \in T} N^2(u)\right).
    \end{equation*}
		(So the definition of $Z$ gets updated at each step, as we add more elements to $T$.)
    Note that, with
    Claim~\ref{clm:second_partition}(\ref{clm:excellent}),
    $|Z| \ge \psi^2 n/2 - |W| - r \beta n \ge \psi^2 n/3$ in both cases.
    When $s_i = 1$, we 
    we add one vertex from $Z$ to $T$,
    and when $s_i = 2$,
    we can add two adjacent vertices in $Z$ to $T$,
    since 
    Claim~\ref{clm:second_partition}(\ref{clm:s_i_equals_2})
    implies that there exists an edge in $G[Z]$.
		
		This completes the construction of $T$. Note that ($\alpha_1$) and ($\alpha_2$) immediately hold. 
		Further, one of ($\beta_1$)--($\beta _4$) holds in each case. It is easy to see that in any case we obtain a set $T$ as desired.
  \end{proof}

  \subsection{Finishing the proof} \label{sec:finishing}
  We now finish the proof by constructing $\mathcal{D}=(D_1, \dots, D_{p+q})$
  a proper ordered collection of $G$ and a
  collection $\T$ of vertex-disjoint $\mathcal{D}$-well-balanced sets 
  as described above.
  We build the collection $\T$ iteratively,
  and, at times, it may include $(r+1)$-sets, as well as $r$-sets.

  Let $c_i := |C_i| - s_i n/r$ for every $i \in [p+q]$.
  For each $i \in [p]$ such that $c_i \ge 0$,
  we will find a $\hat{K}_{s_i+1}$-tiling $\mathcal{S}_i$ of $G[C_i]$
  containing exactly $c_i$ elements. 
  When $s_i = 1$, each vertex in $C_i$ has at least
  \begin{equation*}
    \ceiling{(\delta(G) - 2|\overline{C_i}|)/2} \ge |C_i| - n/r  = c_i
  \end{equation*}
  neighbours in $C_i$.
  Therefore, by Lemma~\ref{lem:matching}, we
  can let $\mathcal{S}_i$ be a matching of size $c_i$ in $G[C_i]$.
  Similarly, for $i \in [p]$ such that $s_i = 2$ and $|C_i| \ge 2n/r$, 
  there exists a matching $M_i$ containing at least 
  \begin{equation*}
    \delta(G) - (n - 1) - |\overline{C_i}| \ge |C_i| - 2n/r=c_i
  \end{equation*}
  heavy edges in $G[C_i]$, and,
  by Claim~\ref{clm:second_partition}(\ref{clm:s_i_equals_2}), 
  we can pair each edge in $e \in M_i$ to a distinct vertex $v_e$ 
  such that $v_e$ is a neighbour of both endpoints of $e$.
  Therefore, we have a collection $\mathcal{S}_i$ of $c_i$ vertex-disjoint 
  elements of $\hat{\K}_3$ in $G[C_i]$.
  We let $\mathcal{S}$ be the union of the sets $\mathcal{S}_i$ constructed so far.

  If $q = 1$ and $|C_{p+1}| > sn/r$, we can use
  Claim~\ref{clm:tolerant_sets}(\ref{clm:large_cliques}) to 
  find  a $\K'_{s+1}$-tiling of size $c_{p+1} = |C_{p+1}| - sn/r$ in $G[C_{p+1}]$,  we call this tiling $\mathcal{S}_{p+1}$ and we add it to $\mathcal{S}$.

  Now suppose $q = 2$ and $|C_{p+1} \cup C_{p+2}| \ge sn/r$;
  we find a tiling consisting of exactly $ |C_{p+1} \cup C_{p+2}|-sn/r$
  copies of $\K'_{s+1}$ in $C_{p+1}$.
  Note that this is trivial to do, by
  Claim~\ref{clm:second_partition}(\ref{clm:excellent}).
  Indeed, we can easily find $|C_{p+1} \cup C_{p+2}| - sn/r$ vertex disjoint
  $(s+1)$-sets in $C_{p+1}$, 
  each forming a clique in the heavy edges of $G[C_{p+1}]$.
  We call this set $\mathcal{S}_{p+1}$ and
  we add the sets in $\mathcal{S}_{p+1}$, to $\mathcal{S}$.
	Note that there is slack in the argument here: given a single fixed set $X \subseteq V$ where $|X|\leq \gamma n$ we can additionally ensure no tile in $\mathcal{S}_{p+1}$ intersects $X$. We will use this property shortly.
	
	We have now completely defined $\mathcal S$. Our next goal is to construct the sets $D_1, \dots, D_{p+q}$ such that $\mathcal{D}=(D_1, \dots, D_{p+q})$ is
  a proper ordered collection of $G$. We will first define the $D_i$ for $i$ such that $c_i\geq 0$. Once we have done this  we will then define the remaining $D_i$.

	Suppose that $q = 2$ and $|C_{p+1} \cup C_{p+2}| \ge sn/r$.
	Remove $c_{p+1}$ vertices from $C_{p+1}$ and call the resulting set $D_{p+1}$; we do this in such a way that these removed vertices consist of precisely one vertex from each tile in $\mathcal S_{p+1}$.
	Place these $c_{p+1}$ vertices into a set $F'$. Let $D_{p+2}:=C_{p+2}$.
This ensures that $|D_{p+1} \cup D_{p+2}| = sn/r$.
  However, we also need to ensure 
  that both $|D_{p+1}|$ and $|D_{p+2}|$ are divisible by $s$.
  Therefore, 
  we will find a set $X \subseteq C_{p+1}$ and a set $Y \subseteq C_{p+2}$
  such that, $|X \cup Y| = s$, 
  $G[X \cup Y]$ contains an element of $\K'_s$, and 
  \begin{equation*}
    |Y| \equiv |C_{p+2}| \pmod s.
  \end{equation*}
  We will then move the vertices in $Y$ from $D_{p+2}$ to $D_{p+1}$.
  We will need the exact minimum degree condition to construct these sets $X,Y$.
	So actually formally what we do is 
first construct $X$ and $Y$  then the collection 
  $\mathcal{S}_{p+1}$ as before such that 
  $\mathcal{S}_{p+1}$ is disjoint from the set $X$.
  Also, for consistency, we will construct the sets $X$ and $Y$ even when
  there is no divisibility issue, i.e.\ 
  when $|C_{p+2}|$ is divisible by $s$.
  In this case, we let 
  $Y = \emptyset$ and $X$ be a clique on $s$ vertices in the heavy edges of $G[C_{p+1}]$,
  which can be found easily using
  Claim~\ref{clm:second_partition}(\ref{clm:excellent}).
  Therefore, it only remains to show how we construct
  $X$ and $Y$ when $|C_{p+2}|$ is not divisible by $s$.
  If $s = 2$, then, for some $i \in \{p + 1, p + 2\}$, $|C_i| \ge n/r$,
  so $\delta(G) - 2(|\overline{C_i}| - 1) \ge 1$, which
  implies that every vertex in $\overline{C_i}$ has a neighbour in $C_i$,
  and, in particular, there exists an edge $xy$
  such that $x \in C_{p+1}$ and $y \in C_{p+2}$,
  and we let $X := \{x\}$ and $Y := \{y\}$.
  If $s = 4$, then for some $i \in \{p+1, p+2\}$, $|C_{i}| \ge 2n/r$,
  and $\delta(G) - (n-1) - (|\overline{C_i}| - 1) \ge 1$,
  which implies that every vertex in $\overline{C_i}$ has a heavy
  neighbour in $C_i$ and we can let $x_1y_1$ be a heavy edge
  such that $x_1 \in C_{p+1}$ and $y_1 \in C_{p+2}$.
  Let $j$ be such that $1 \le j \le 3$ and 
  $j \equiv |C_{p+2}| \pmod 4$.
  By Claim~\ref{clm:second_partition}(\ref{clm:excellent})--(iii),
  we can find vertices $y_2, \dotsc, y_j$ in $N(x_1) \cap (C_{p+2} - y_1)$, 
  such that $Y = \{y_1, y_2, \dotsc, y_j\}$ induces a clique on heavy edges, 
  and then find vertices $x_2, \dotsc, x_{4 - j}$
  in $N(y_1) \cap \dotsm \cap N(y_j) \cap (C_{p+1} - x_1)$
  such that $X = \{x_1, \dotsc, x_{4-j}\}$ induces a clique on heavy edges.
  Note that $G[X \cup Y]$ contains an element of $\K'_4$.

We have described how to define $D_{p+1}$ and $D_{p+2}$ in the case when $q=2$ and $c_{p+1}\geq 0$. We now describe in general how to construct $D_i$  when $c_i \geq 0$.
  Using 
  Claim~\ref{clm:extend_partial_clique}(\ref{clm:extend_partial_clique_std}), 
  we can find, for each set in 
  $S \in \mathcal{S}$,
  a set $T'$ such that $G[T' \cup S]$ contains an element of $\K'_{r+1}$
  and when we arbitrarily select a vertex $v \in S$,
  the set $T' \cup (S - v)$ is $\mathcal{C}$-balanced.
  We let $F'$ be the set of these arbitrarily selected vertices.
  Recall that when $q = 2$, the set $S$ does not intersect
  $C_{p+2}$, so
  $T' \cup (S - v)$ is actually $\mathcal{C}$-well-balanced.
  We label $T' \cup S$ as $T_v$ and $S$ as $S_v$.
  By Claim~\ref{clm:extend_partial_clique}(\ref{clm:extend_partial_clique_std}), 
  we can assume that, for every $v \in F'$,
  the sets $T_v$ were constructed so as to be vertex-disjoint and,
  when $q = 2$, disjoint from $X \cup Y$.
  For every $v \in F'$, we add $T_v$ to $\T$.
  When $q = 2$,
  using 
  Claim~\ref{clm:extend_partial_clique}(\ref{clm:extend_partial_clique_std}), 
  we find a set $T'$, disjoint from all of the previously constructed
  sets, such that
  $G[T' \cup X \cup Y]$ contains an element of $\K'_r$ and 
  $|T' \cap C_i| = s_i$ for each $i \in [p]$,
  and we add $T' \cup X \cup Y$ to $\mathcal{T}$.
  We now let 
  \begin{equation*}
    D_i := C_i \setminus F' \text{ for all $i \in [p+q]$},
  \end{equation*}
  and note that for every $i \in [p]$, if $c_i \ge 0$, then $|D_i| = s_i n/r$. 
  Furthermore, when $q \ge 1$ and 
  $|C_{p+1} \cup C_{\sigma(p+1)}| \ge s n/r$, we have that
  $|D_{p+1} \cup D_{\sigma(p+1)}| = s n/r$, and, 
  when $q = 2$, we also have that both
  $|D_{p+1}|$ and $|D_{\sigma(p+1)}|$ are divisible by $s$.


  For each vertex $v \in \overline{C}$, we use 
  Claim~\ref{clm:extend_overline_C}(\ref{clm:extend_overline_C_std})
  to find a vertex set $T_v$ that is $\mathcal{B}$-well-balanced
  and such that $G[T_v + v]$ contains a $\K'_{r+1}$.
  We add $T_v + v$ to $\T$ and ensure that
  these sets are disjoint from the sets already in $\T$.

  Let $F := F' \cup \overline{C}$ and recall that,
  for every $i \in [p]$ such that $|C_i| \le s_in/r$, 
  we currently have that $D_i = C_i$.
  At this point, every vertex in $V$ is either in one of the sets
  $\{D_1, \dotsc, D_{p+q}\}$ or is in $F$.
  
  We now move vertices from $F$ to sets in $\mathcal{D}$ 
  that are ``too small'' until we have the
  desired proper ordered collection.
  When we do this we also make small changes to the collection $\T$ 
  so that every $T \in \T$ will be a $\mathcal{D}$-well-balanced $r$-set.
	
  In detail, for every $v \in F$ and $i \in [p+q]$, when we say we 
  \emph{assign $v$ to $D_i$} we mean that we add $v$ to the set $D_i$
  and update $T_v$ by removing one $u \in C_i$ from $T_v$ and add $v$ to $T_v$.
	This is only well-defined when there is initially some $u \in C_i$ in $T_v$.
	Note that in this case the updated version of $T_i$ is $\mathcal{D}$-well-balanced.
	If initially $T_v$ contains no element from $C_i$ then 
	 $q = 2$, $i \in \{p+1, p+2\}$, and $T_v$ 
  intersects $C_{\sigma(i)}$ instead of $C_i$.
  Furthermore, since we are moving a vertex
  to either $C_{p+1}$ or $C_{p+2}$ from $F$, 
  it must be that $|C_{p+1} \cup C_{p+2}| < sn/r$, and so if $v \in F'$ then\COMMENT{AT: added ` so if $v \in F'$ then'}
  $S_v \subseteq C_i$ for some $i \in [p]$ such that $c_i > 0$.
  In this case when we assign $v$ to $D_i$ we instead complete the following process:
We first remove $T_v$ from $\T$.
Then, if $v \in \overline{C}$,
  we use 
  Claim~\ref{clm:extend_overline_C}(\ref{clm:extend_overline_C_special}), 
  to find a set $T' \subseteq B \setminus V(\T)$
  such that $G[T '+ v]$ contains
  an element of $\K'_r$,
  $|T' \cap B_j| = s_j$
  for every $j \in [p]$, 
  $|T' \cap B_i| = s - 1$,
  $|T' \cap B_{\sigma(i)}| = 0$. Then
 let $T_v := T' + v$ and
  add it $\T$.
  Similarly, when $v \in F'$, 
  we use
  Claim~\ref{clm:extend_partial_clique}(\ref{clm:extend_partial_clique_special}) 
  to find a set $T' \subseteq C \setminus V(\T)$ 
  such that $G[T' \cup S_v]$ contains an element of $\K'_r$,
  $|T' \cap C_k| = 0$ where $k \in [p]$ such that $S_v \subseteq C_k$,
  $|T' \cap C_j| = s_j$
  for every $j \in [p] - k$, 
  $|T' \cap C_i| = s - 1$,
  $|V(T') \cap C_{\sigma(i)}| = 0$,
  and let $T_v := T' \cup S_v$ \COMMENT{AT:changed def of $T_v$ here} and
  add it $\T$.
  We then add $v$ to $D_i$. Note that now $T_v$ is $\D$-well-balanced.

  For any $i \in [p+q]$ such that $c_i < 0$, 
  we then arbitrarily assign exactly $-c_i$ of the 
  remaining vertices in $F$ to $D_i$, 
  except when $q = 2$ and $i \in \{p+1, p+2\}$.
  When $q = 2$ and $|C_{p+1} \cup C_{p+2}| < sn/r$, 
  we again have to be careful to ensure that, in the end,
  both $|D_{p+1}|$ and $|D_{p+2}|$ are divisible by $s$.
  Therefore, assume that, when $q = 2$, we assign vertices from 
  $F$ to $D_{p+1}$ and $D_{p+2}$ 
  before we assign vertices in $F$ to any $D_i$ for $i \le p$.
  Also, note that, because $|D_{p+1}  \cup D_{p+2}| < sn/r$, $|F| \ge 1$.
  To help us organise the assignment of vertices, 
  we let $j \in \{p+1, p+2\}$,
  so that if we let $1 \le k_j, k_{\sigma(j)}\le s$ be such that
  \begin{equation*}
    k_j \equiv |D_j| \text{ and }
    k_{\sigma(j)} \equiv |D_{\sigma(j)}| \pmod s, 
  \end{equation*}
  then $k_{\sigma(j)} \le k_j$.
  If $|F| \ge s - k_j$, we assign $s - k_j$ vertices in $F$ to $D_{j}$ 
  and then assign vertices from $F$ to $D_{\sigma(j)}$ until 
  $|D_{p+1} \cup D_{p+2}| = sn/r$.
  Otherwise, we can assume $1 \le |F| < s - k_j$, which
  implies $s = 4$, and $|F| \le 2$.
  Note that if $|F| = 1$, then 
  $|F \cup D_{p+1} \cup D_{p+2}| = 4n/r$,
  so $k_j + k_{\sigma(j)} = 3$.
  Therefore, exactly one of the two following conditions must hold:
  \begin{equation*}
   (i) \ |F| = 2 \text{ and } k_j = k_{\sigma(j)} = 1; \text{ or }
   (ii) \ |F| = 1, k_j = 2 \text{ and } k_{\sigma(j)} = 1. 
  \end{equation*}
  In either case, we arbitrarily pick $v \in F$,
  add $v$ to $D_{\sigma(j)}$ and delete $T_v$ from $\T$.
  We then use Claim~\ref{clm:extend_overline_C}(\ref{clm:extend_overline_C_special}) (if $v \in \overline{C}$)
  or Claim~\ref{clm:extend_partial_clique}(\ref{clm:extend_partial_clique_special}) (if $v \in F'$),
  to construct a $\D$-balanced set $T_v$ containing $v$ that has
  one vertex in $C_{\sigma(j)}$ and two vertices in $C_{j}$. 
  We then move both of the vertices in $T_v \cap D_{j}$ from
  $D_{j}$ to $D_{\sigma(j)}$.
  In both cases, we now have that
  $|D_{\sigma(j)}|$ is divisible by $4$,
  and $T_v$ is $\mathcal{D}$-well-balanced.
  We then assign the possibly one remaining vertex in $F$
  to $D_{j}$, so $|D_{j}|$ is divisible by $4$ as well.
  
  For every $i \in [p+q]$, we let $D_i' = D_i \setminus V(\T)$, 
  and note that $D'_i \subseteq C_i$.
  Therefore, if $G' = G - V(\T)$, then 
  $\mathcal{D}' = (D'_1, \dotsc, D'_{p+q})$ is a 
  proper ordered collection of $G'$, so, by Claim~\ref{clm:proper_partitions}
  there exists a perfect $\K'_r$-tiling $\T'$ of $G'$,
  and $\T \cup \T'$ is a perfect $\K'_r$-tiling of $G$.

\section*{Acknowledgements}
This research was partially carried out whilst the last three authors were visiting the Institute for Mathematics and its Applications at the University of Minnesota.
The authors would like to thank the institute for the nice working environment. The last two authors were also supported by the BRIDGE strategic alliance between the University of Birmingham and the University of Illinois at Urbana-Champaign, as part of 
the `Building Bridges in Mathematics' BRIDGE Seed Fund project.

{\footnotesize \obeylines \parindent=0pt

  Andrzej Czygrinow
  School of Mathematics and Statistics
  Arizona State University
  Tempe, AZ 85281
  US
\tt{aczygri@asu.edu}}

{\footnotesize \obeylines \parindent=0pt

  Louis DeBiasio
  Department of Mathematics
  Miami University
  Oxford, OH 45056
  US
\tt{debiasld@miamioh.edu}}

{\footnotesize \obeylines \parindent=0pt

  Theodore Molla 
  Department of Mathematics
  University of Illinois at Urbana-Champaign
  Urbana, IL 61801
  US
\tt{molla@illinois.edu}}

{\footnotesize \obeylines \parindent=0pt
  Andrew Treglown
  School of Mathematics
  University of Birmingham
  Edgbaston
  Birmingham
  B15 2TT
  UK
\tt{a.c.treglown@bham.ac.uk}}







\begin{thebibliography}{10}
  \bibitem{alon} N. Alon and A. Shapira, Testing subgraphs in directed graphs, 
    \emph{J. Comput. System Sci.} {\bf 69} (2004),
    354--382.
\bibitem{alony} N.~Alon and R.~Yuster,
Almost $H$-factors in dense graphs, \emph{Graphs Combin.} {\bf 8} (1992), 95--102.
  \bibitem{blssw} F.S. Benevides, T. {\L}uczak, A. Scott, J. Skokan\ and M. White,
    Monochromatic cycles in $2$-coloured graphs, 
    \emph{Combin. Probab. Comput.}~{\bf 21} (2012), 57--87.
   
  \bibitem{blm}
    J. Balogh, A. Lo and T. Molla, Transitive triangle tilings in oriented graphs, submitted.
  
  \bibitem{bh} 
   W.G. Brown and F. Harary, Extremal digraphs, \emph{Coll. Math. Soc. J. Bolyai} {\bf 4} (1969), 135--198.  
   
  \bibitem{corradi} K. Corr\'adi and A. Hajnal, On the maximal number of independent circuits in a graph, \emph{Acta Math. Acad. Sci.
    Hungar.}~{\bf 14} (1964), 423--439.

  \bibitem{cdkm} A. Czygrinow, L. DeBiasio, H.A. Kierstead and T. Molla, An extension of the Hajnal--Szemer\'edi theorem to directed graphs, \emph{Combin. Probab. Comput.}~{\bf 24} (2015), 754--773.
  \bibitem{ckm} A. Czygrinow, H.A. Kierstead and T. Molla, On directed versions of the Corr\'adi--Hajnal Corollary, \emph{European J. Combin.}~{\bf 42} (2014), 1--14.
\bibitem{hlad} C. Grosu and J. Hladk\'y, The extremal function for partial bipartite tilings, \emph{
European J. Combin.}
{\bf 33}  (2012),  807--815.
  \bibitem{hs} A. Hajnal and E.~Szemer\'{e}di, Proof of a conjecture of Erd\H{o}s,
    \emph{Combinatorial Theory and its Applications vol. II}~{\bf 4} 
    (1970), 601--623.
  \bibitem{ht00}F. Havet, S. Thomass\'{e}, Oriented hamiltonian paths in
    tournaments: A proof of Rosenfeld's conjecture, \emph{J.
    Combin. Theory B} \textbf{78}, (2000), 243--273.
  \bibitem{hell} P. Hell and D.G. Kirkpatrick, On the complexity of general graph
    factor problems, \emph{SIAM J. Computing}~{\bf 12} (1983), 601--609.
  \bibitem{Janson&Luczak&Rucinski00} S. Janson, T. {\L}uczak\ and\ A. Ruci\'nski,
    {\em Random Graphs}, Wiley, 2000.
  \bibitem{my} P. Keevash and R. Mycroft, A multipartite Hajnal--Szemer\'edi theorem, \emph{J. Combin. Theory~B} {\bf 114} (2015), 187--236. 
  \bibitem{keevs} P. Keevash and B. Sudakov, Triangle packings and 1-factors in oriented graphs, \emph{J. Combin. Theory~B}~{\bf 99} (2009), 709--727. 
  \bibitem{kier} H.A. Kierstead and A.V. Kostochka, An Ore-type Theorem on Equitable
    Coloring, \emph{J. Combin. Theory B}~{\bf{98}}  (2008), 226--234.
  \bibitem{short} H.A. Kierstead and A.V. Kostochka, A short proof of the Hajnal--Szemer\'edi Theorem on Equitable
    Coloring, \emph{Combin. Probab. Comput.}~{\bf{17}}  (2008), 265--270.
\bibitem{ko} J. Koml\'{o}s, Tiling Tur\'{a}n Theorems, \emph{Combinatorica}~{\bf 20}
(2000), 203--218.
     \bibitem{komlos2001proof}
    J.~Koml{\'o}s, G.~S{\'a}rk{\"o}zy, and E.~Szemer{\'e}di, {Proof of the
    Alon--Yuster conjecture}, \emph{Disc. Math.}~\textbf{235} (2001), 255--269.
		
		\bibitem{ks}
    J.~Koml{\'o}s and M.~Simonovits, 
    Szemer{\'e}di's {R}egularity {L}emma and its applications in graph theory. 
    \emph{Bolyai Society Studies 2, Combinatorics, {P}aul {E}rd\H{o}s is {E}ighty \textbf{2}} (1996), 295--352.

  \bibitem{kuhn} D. K\"{u}hn and D. Osthus, Critical chromatic number and the
    complexity of perfect packings in graphs, \emph{17th ACM-SIAM Symposium on 
    Discrete Algorithms} (SODA 2006), 851--859.
  \bibitem{kuhn2} D. K\"{u}hn and D. Osthus, The minimum degree threshold for perfect
    graph packings, \emph{Combinatorica}~{\bf 29} (2009), 65--107.



  \bibitem{survey} D. K\"uhn and D. Osthus, Embedding large subgraphs into dense graphs, in \emph{Surveys in Combinatorics}
    (S. Huczynska, J.D. Mitchell and C.M. Roney-Dougal eds.),
    \emph{London Math.~Soc.~Lecture Notes}~\textbf{365}, 137--167, Cambridge University Press, 2009.


  \bibitem{lo} A. Lo and K. Markstr\"om, $F$-factors in hypergraphs via absorption, \emph{Graphs Combin.} {\bf 31} (2015), 679--712.
		  \bibitem{molla13phd}
    T.~Molla, \emph{Tiling directed graphs with cycles and tournaments}, Ph.D.
    thesis, Arizona State University, Tempe, Arizona, 2013.
		\bibitem{rrs2} V. R\"odl, A. Ruci\'nski and E. Szemer\'edi, A Dirac-type theorem for 3-uniform hypergraphs,
\emph{Combin. Probab. Comput.}~{\bf 15} (2006), 229--251.
  \bibitem{reglem} E.~Szemer\'edi, Regular partitions of graphs, 
\emph{Probl\'emes Combinatoires et Th\'eorie des Graphes Colloques Internationaux
CNRS} {\bf 260} (1978), 399--401.
  
  \bibitem{problem} A. Treglown, A note on some oriented graph embedding problems, \emph{J. Graph Theory}~{\bf 69} (2012), 330--336.



  \bibitem{treg} A. Treglown, On directed versions of the Hajnal--Szemer\'edi theorem, \emph{Combin. Probab. Comput.}~{\bf 24} (2015), 873--928.
\bibitem{triangle} A.~Treglown,
A degree sequence Hajnal-Szemer\'edi theorem,
\emph{J. Combin. Theory~B} {\bf 118} (2016), 13--43.
  \bibitem{wang} H. Wang, Independent directed triangles in a directed graph, \emph{Graphs  Combin.}~{\bf 16} (2000),
    453--462.
  \bibitem{zhang} H. Wang and D. Zhang, Disjoint directed quadrilaterals in a directed graph, \emph{J. Graph Theory}~{\bf 50} (2005), 91--104.
  \bibitem{yuster} R. Yuster, Tiling transitive tournaments and their blow-ups, \emph{Order}~{\bf 20} (2003), 121--133.
  \bibitem{zsurvey} Y. Zhao, Recent advances on Dirac-type problems for hypergraphs, \emph{Recent Trends in Combinatorics}, the IMA Volumes in Mathematics and its Applications 159. Springer, New York, 2015. Vii 706. 
\end{thebibliography}
\end{document}